\documentclass[10pt,twoside,reqno,a4paper]{article}

\usepackage[T1]{fontenc}
\usepackage{a4wide}

\usepackage{amsfonts, amsmath, amssymb,latexsym}
\usepackage[activeacute,english]{babel}
\usepackage{epsfig}
\usepackage{graphicx}
\usepackage[font=sf,sc]{caption}
\usepackage{float}
\usepackage{cancel}
\usepackage{color}
\usepackage[utf8]{inputenc}
\usepackage{amsthm}
\usepackage{bm}
\usepackage[all]{xy}
\usepackage{enumerate}
\usepackage{comment}
\usepackage[colorinlistoftodos]{todonotes} 
\usepackage{bbm}

\usepackage[colorlinks=true,linkcolor=blue,urlcolor=blue,citecolor=blue]{hyperref}

\newcommand{\R}{{\mathbb {R}}}

\renewcommand{\leq}{\leqslant}
\renewcommand{\geq}{\geqslant}

\newcommand{\RR}{\mathbb{R}}

\newcommand{\DDD}{\mathcal{D}}

\newcommand\cadlag{c\`{a}dl\`{a}g}

\newcommand{\Scal}[1]{ {}_q\langle #1 \rangle_p}
\newcommand{\ScalD}[1]{\left  \langle #1 \right \rangle_{\DDD}}
\newcommand{\IndN}[1]{\mathbbm{1}_{#1}}

\newtheorem{theorem}{Theorem}

\theoremstyle{plain}

\newtheorem{corollary}[theorem]{Corollary}

\newtheorem{definition}[theorem]{Definition}
\newtheorem{example}[theorem]{Example}

\newtheorem{lemma}[theorem]{Lemma}

\newtheorem{proposition}[theorem]{Proposition}
\newtheorem{prop}[theorem]{Proposition}
\newtheorem{remark}[theorem]{Remark}

\numberwithin{equation}{section}
\numberwithin{theorem}{section}
\newcommand{\E}{\ensuremath{\mathbb{E}}}

\allowdisplaybreaks

\newtheorem*{assumption}{Assumption}
\theoremstyle{definition}

\usepackage{enumerate}

\definecolor{mycolor}{rgb}{1,0.5,0.5}

\title{Stochastic systems with memory and jumps}

\author{D.R. Ba\~{n}os\thanks{Institute of Mathematics of the University of Barcelona,
University of Barcelona, Gran Via de les Corts Catalanes 585, 08007, Barcelona, and Department of Mathematics, University of Oslo, P.O. Box 1053 Blindern, N-0316 Oslo, Norway, Email: davidru@math.uio.no}, 
F. Cordoni\thanks{Department of Mathematics, University of Trento, Via Sommarive, 14, 38123 Trento, Italy. Email: francesco.cordoni@unitn.it}, 
G. Di Nunno\thanks{Department of Mathematics, University of Oslo, P.O. Box 1053 Blindern, N-0316 Oslo, Norway, and Norwegian School of Economics and Business Administration, Helleveien 30, N-5045 Bergen, Norway. Email: giulian@math.uio.no}, 
L. Di Persio\thanks{Department of Computer Science, University of Verona, Strada le Grazie, 15, 37134 Verona, Italy. Email: luca.dipersio@univr.it }, and 
E.E. R\o se\thanks{Department of Mathematics, University of Oslo, P.O. Box 1053 Blindern, 0316 Oslo, Norway. Email: elinero@math.uio.no}}
\date{\today}

\begin{document}
\maketitle

\begin{abstract}
Stochastic systems with memory naturally appear in life science, economy, and finance. We take the modelling point of view of stochastic functional delay equations and we study these structures when the driving noises admit jumps. Our results concern existence and uniqueness of strong solutions, estimates for the moments and the fundamental tools of calculus, such as the It\^o formula. We study the robustness of the solution to the change of noises. Specifically, we consider the noises with infinite activity jumps versus an adequately corrected Gaussian noise. The study is presented in two different frameworks: we work with random variables in infinite dimensions, where the values are considered either in an appropriate $L^p$-type space or in the space of c\`adl\`ag paths. The choice of the value space is crucial from the modelling point of view as the different settings allow for the treatment of different models of memory or delay. Our techniques involve tools of infinite dimensional calculus and the stochastic calculus via regularisation. 


\end{abstract}


\section{Introduction}

Delay equations are differential equations whose coefficients depend also on the past history of the solution. Besides being of mathematical interest on their own, delay equations naturally arise in many applications, ranging from mathematical biology to mathematical finance, where often the effect of the memory or delay on the evolution of the system cannot be neglected, we refer to \cite{Chang2,Chang,MR2241374,MR2332878,MR1218880,kuchler,MR1652338} and references therein for applications in different areas.

When dealing with a delay differential equation (DDE), one cannot in general relay on standard existence and uniqueness theorems, but ad hoc results have to be proven. In general this is done by lifting the DDE, from having a solution with values in a finite dimensional state space, such as $\R^d$, to having values in an infinite dimensional path space, which has to be carefully chosen according to the specific problem. For the case of \emph{deterministic} delay differential equations an extensive literature exists, we refer the reader to the monographs \cite{diekmann1995delay,engel2000one} for details.

When considering \emph{stochastic} delay differential equations (SDDE), that is DDE perturbed by a stochastic noise, one encounters problems that did not appear in the deterministic case  or in classical stochastic differential equations. In particular the SDDE fails to satisfy the Markov property, hence one cannot rely on the well established setting of Markov processes for the study of the solution. As in the deterministic case, however, one can apply the key idea to lift the SDDE to have values in a suitable infinite dimensional path space. In doing so, one is able to recover the Markov property, nevertheless the main drawback is that now one is dealing with an infinite dimensional stochastic partial differential equation (SPDE). Although a well established theory for SPDE's exists, some fundamental results, known in the finite dimensional case, fail to hold true in the infinite dimension. In particular when considering infinite dimensional SPDE's, the concept of \emph{quadratic variation} is not a straightforward generalisation of the classical notion of quadratic variation. We recall that this concept is crucial in essential tools of stochastic analysis, such as the It\^o formula.
Some concrete results around the concept of quadratic variation in infinite dimensions have appeared only recently, see \cite{MR3272616}. 

SDDE's have been first studied in the seminal works \cite{ChMi,SMohammed2} and then extensively studied in \cite{f10st,MR1652338,Yan}, though in a different, but yet related setting. Recently there has been a renewed interest in SDDE's motivated by financial applications. In \cite{dupire2009functional} a path dependent stochastic calculus was first suggested and then widely developed in \cite{cont2010change,cont2013functional}.

From the stochastic calculus point of view, the technique of regularisation, recently introduced, proved to be powerful to define the a stochastic integral and to prove a general It\^{o} formula for stochastic differential equation both in finite and infinite dimensions. The first results exploiting the stochastic calculus via regularisation are found in \cite{MR1350257,MR1367665}, where a generalisation of the It\^{o} formula was proved. More recently in \cite{di2012generalized} a new  concept of quadratic variation for Banach space-valued processes was proposed and applied to prove a suitable It\^{o} formula for infinite dimensional stochastic processes. This triggered a stream of studies aimed at deriving a suitable It\^{o}'s formula for delay equations and at studying deterministic problems that can be tackled by a stochastic approach. Particular attention was given to the Kolmogorov equation. We refer to \cite{c14cal,c15fun,c14reg,f13in,f10st}. Eventually in \cite{c15fun,f13in} the relationship between the \textit{path-wise calculus} and the \textit{Banach-space} calculus was detailed.

We remark that all the aforementioned results for delay equations are proved in the case when the driving noise is continuous, such as for a standard Brownian motion. Very few results exist when the noise allows for random jumps to happen, see, e.g. \cite{reiss2002nonparametric,MR2260741}. 

As mentioned above, there are different approaches to deal with SDDE's. 
We work with the setting of stochastic functional delay differential equations (SFDDE), first introduced in \cite{SMohammed2} and further developed in \cite{MR1652338,Yan}. This choice is motivated by the fact that it appears to be the right compromise between a general purely infinite dimensional SPDE and a classical finite dimensional SDE. In fact, even if the stochastic delay equation is treated as an infinite dimensional equation, there is always a clear connection with its finite dimensional realisations, so that standard finite dimensional It\^{o} calculus can be often used.

The aim of the present paper is to extend the theory of SFDDE, studied in \cite{SMohammed2} for a Brownian driving noise, to include jumps, that is to deal with noises of jump-diffusion type. Specifically we aim at settling the existence and unicity of solutions, and derive the fundamental tools of a stochastic calculus for SFDE's with jumps. We also study the robustness of the solutions of SFDDEs to changes of the driving noise. This is an important analysis in view of the future applications. From a finite dimensional point of view this was studied in e.g. \cite{MR2814479}.

We consider an $\R^d-$valued SDE of the form
\begin{equation}\label{initialproblem}
\begin{split}
d X(t)  = & f(t,X(t),X(t+\cdot)) dt + g(t,X(t),X(t+\cdot)) dW(t) \\ 
 & + \int_{\R_0} h(t,X(t),X(t+\cdot))(z) \tilde{N}(dt,dz)\, ,\quad t \in [0,T],
\end{split}
\end{equation}
where $W$ is a standard Brownian motion, $\tilde{N}$ is a compensated Poisson random measure and $f$, $g$ and $h$ are some given suitable functional coefficients. With the notation $X(t+\cdot)$ we mean that the coefficient may depend also on the past values of the solution on the interval $[t-r,r]$ for some fixed delay $r>0$. 
It is this dependence on the past values of the evolution that is identified as memory or, equivalently, delay.
The formal introduction of the current notation will be carried out in the next section. Notice that at this stage equation \eqref{initialproblem} is a finite dimensional SDE with values in $\R^d$.

We now lift the process \eqref{initialproblem} to have values in a suitable infinite dimensional path space. The choice of the suitable space is truly a key issue. As illustration, consider the purely diffusive case and denote the maximum delay appearing in \eqref{initialproblem} by $r >0$. Then, in \cite{Yan} a product space of the form $M^p :=L^p([-r,0] ;\R^d) \times \R^d$, $p \in [2,\infty)$, was chosen, whereas in \cite{MR1652338} the space of continuous functions $\mathcal{C}:= \mathcal{C}([-r,0];\R^d)$ was taken as reference space. With the former choice one can rely on well-established results and techniques for $L^p-$spaces. Nevertheless this choice may seem artificial when dealing with a past path-dependent memory. For this reason the second choice, the space of continuous functions, is often considered the right space where to study delay equations, though it requires mathematically careful considerations. 

The natural extension of SFDDE's to the jump-diffusion case correspondingly leads to two possible choices of setting: the product space $ M^p$ and the space of \cadlag \ (right continuous with finite left limit) functions $\DDD := \DDD([-r,0];\RR^d)$. We decided to carry out our study in both settings in order to give a general comprehensive and critical presentation of when and in what sense it may be more suitable to treat the study in the one or the other setting. In fact, on the one side, we have the inclusion $\DDD \subset M^p$, with the injection being continuous, so that the $M^p-$setting appears to be more general, on the other side we see that existence and uniqueness of the solution of an SFDDE cannot be established in general in the space $M_p$. This in fact depends on the choice of type of delay or memory. In fact the drawback of the $M^p$ approach is that it does not apply to 
to SFDDE's with discrete delay, which are equations of the form
\begin{align}
 X(t)=\int_0^t X(s+\rho)ds+\int_0^t X(s+\rho)dW(s)\, ,
\end{align}
where $\rho \in [-r,0)$ is a fixed parameter. Certainly, the reason is that $X(t+\rho)$ is actually interpreted as an evaluation of the segment $X_t =\{X(t+s), s\in [-r,0]\}$ at the point $s=\rho$. Such operation is not well-defined in the $M^p$-setting. To see this, simply take two elements $(\eta_1,\eta_1(0)),(\eta_2,\eta_2(0))\in M^p$ such that $\eta_1(s) = \eta_2(s)$ for all $s\in [-r,0]\setminus\{\rho\}$ and $\eta_1(\rho)\neq \eta_2(\rho)$. Then, clearly $(\eta_1,\eta_1(0))$ and $(\eta_2,\eta_2(0))$ belong to the same class in $M^p$, but the evaluation at $\rho$ is not uniquely determined. The discrete delay case can be treated in the setting given by $\DDD$.

In the sequel, we thus lift equation \eqref{initialproblem} to have values either in $M^p$, with $p \in [2,\infty)$, or in $\DDD$, exploiting the notion of \textit{segment}. We then study a SFDDE of the form,
\begin{align}\label{eq:SFDEgeneral}
d X(t)&=f(t, X_t) d t+g(t,X_t)d W(t)+ \int_{\mathbb{R}_0}h(t,X_t)(z)\tilde N(d t, d z)\\
X_0&=\eta \notag 
\end{align}
where we denote $X_t$ the \textit{segment} of the process $X$ on an interval $[t-r,t]$, that is 
\[
X_t := \left \{X(t+\theta) \, : \, \theta \in [-r,0] \right \}\, ,
\]
being $r\geq 0$ the maximum delay. By $X(t)$ we denote the present value of the process at time $t$. Also $\eta$ is a function on $[-r,0]$. 

In this paper, first we establish existence, uniqueness and moment estimates for the equation \eqref{eq:SFDEgeneral} where the \textit{segment} $X_t$ takes values either in $M^p$ or in $\DDD$.
Then we look at the robustness of the model to changes of the noise. In particular, we study what happens if we replace the small jumps of the infinite activity Poisson random measure $N$, by a continuous Brownian noise $B$. This is done by comparing a process $X$ with dynamics
\begin{align}
 d X(t)&=f(t, X_t) d t+g(t,X_t)d W(t)+ \int_{\mathbb{R}_0}h_0(t,X_t)\lambda(z)\tilde N(d t, d z)\\
X_0&=\eta, \notag
\end{align}
to the process $X^{(\epsilon)}$ defined by
\begin{align}
d X^{(\epsilon)}(t)&=f(t, X^{(\epsilon)}_t) d t+g(t,X^{(\epsilon)}_t)d W(t)+h_0(t,X^{(\epsilon)}_t)\int_{|z|<\epsilon}|\lambda(z)|^2 \nu(dz) dB(t)\notag \\
&+ \int_{|z|\geq \epsilon}h_0(t,X^{(\epsilon)}_t)\lambda(z)\tilde N(d t, d z)\\
X^{(\epsilon)}_0&=\eta \, .\notag
\end{align}
We remark that the choice of this approximate guarantees the same so-called total volatility, using a terminology from financial modelling. 

Eventually, exploiting the \textit{stochastic calculus via regularisation} we prove an It\^{o} type formula for stochastic delay equations with jumps, showing that the results are in fact coherent with the results obtained in \cite{MR1652338,Yan}. We work with forward integrals in the following sense. In general, given the stochastic processes $X = \{X_s, s \in [0,T]\}$, and $Y=\{Y_s,  s \in [0,T]\}$, taking  values in $L^p([0,T], \mathbb{R}^d)$ and its topological dual, respectively, we define the \textit{forward integral} of $Y$ against $X$ as 
\begin{equation}\label{EQN:ForIntDef}
\int_0^t \Scal{Y_s\, , \, dX_s}:= \lim_{\epsilon \downarrow 0} \int_0^t \left  \langle Y_s\, , \, \frac{X_{s+\epsilon}-X_s}{\epsilon} \right  \rangle_p ds\, ,
\end{equation}
where the limit holds in \textit{probability} and we denoted by $\langle \cdot , \cdot \rangle_p$ the paring between $L^p([-r,0], \mathbb{R}^d)$ and its dual. Furthermore, if the above limit holds \textit{uniformly on probability on compact sets} (ucp), we immediately have that the process
\[
\left (\int_0^t \Scal{Y_s\, , \, dX_s}\right )_{t \in [0,T]}\, ,
\]
admits a \cadlag \ version and we say that the \textit{forward integral} exists. When the two processes have values in the space $M^p$ and ${M^p}^*$, we are able to show that the above limit holds in fact ucp, characterizing thus the \textit{forward integral} in terms of the derivative of the process $X$, which coincides with the operators introduced in \cite{Yan} and in \cite{f13in}. 

The present work is structured as follows. In Section \ref{TheEquationsExistenceUniquenessMomentestimates3} we introduce the main notation used throughout the whole paper. In Section \ref{sec:functionalsOnS} we study existence and uniqueness results for equation \eqref{eq:SFDEgeneral} with values in $\DDD$, whereas in Section \ref{sec:M2existenceAndUniqueness} we prove the same results in the $M^p$ setting. Then, in Section \ref{robust} we prove the robustness of equation \eqref{eq:SFDEgeneral} to the change of the noise. Eventually in Section \ref{SEC:Ito} we prove a suitable It\^{o}-type formula for SFDDE's with values in $M^p$ and in $\DDD$.


\section{Stochastic functional differential equations with jumps}\label{TheEquationsExistenceUniquenessMomentestimates3}

\subsection{Notation}\label{TheEquationsExistenceUniquenessMomentestimates2}
Let $(\Omega, \mathcal F, \mathbb F:=\{\mathcal F_t\}_{t\in[0,T]}, P)$ be a complete, filtered probability space satisfying the usual hypotheses for some finite time horizon $T<\infty$. Let $r\geq 0$ be  a non-negative constant denoting the maximum delay of the equations considered. We extend the filtration by letting 
 $\mathcal F_s=\mathcal F_0$ for all $s\in[-r,0]$. This will still be denoted by $\mathbb{F}$.

Let $W=(W^1,\dots, W^m)^{\mathsf T}$ be an $m$-dimensional $\mathbb F$-adapted Brownian motion and  $N=(N^1,\dots, N^n)^{\mathsf T}$ be the jump measures associated with $n$ independent  $\mathbb F$-adapted L\'evy processes, with L\'evy measures $\nu = (\nu_1,\dots ,\nu_n)$ respectively. We denote by $\tilde N$, the compensated Poisson random measure 
\[
\tilde N(d t, d z):=(N^1(d t, d z)-\nu_1(d z)dt, \dots, N^n(d t, d z)-\nu_n(d z)dt)^{\mathsf T} \, .
\]

Consider the equation
\begin{align}\label{eq:Main1}
\begin{split}
 d X(t)&=f(t, X_t) d t+g(t,X_t)d W(t)+ \int_{\mathbb{R}_0}h(t,X_t)(z)\tilde N(d t, d z)\\
X_0&=\eta, 
\end{split} 
\end{align}
where $f$,$g$ and $h$ are some given functionals on a space containing the segments $X_t$, $t\in[0,T]$ of the process $X$. We will give precise definitions of the segments and the coefficient functionals 
below. Equations of the form \eqref{eq:Main1} will be referred to as \emph{stochastic functional delay differential equation} (SFDDE).  

We remark that equation \eqref{eq:Main1} is to be interpreted component-wise as a system of SFDDE of the following form:

\begin{align*}
dX^i(t)&=f^i(t,X_t)dt+ \sum_{j=1}^m g^{i,j}(t,X_t)dW^j (t)+ \sum_{j=1}^n\int_{\mathbb{R}_0}h^{i,j}(t,X_t,z)\tilde N^j(dt,dz),\\
X^i_0&=\eta^i \, , \quad i = 1, \dots, d \, .
\end{align*}

With the component-wise interpretation in mind, it is natural to require that the images $f(t,X_t)$ and $g(t,X_t)$  of the coefficient functionals $f$ and $g$ are contained in the spaces $L^2(\Omega, \mathbb R^d)$ and $L^2(\Omega, \mathbb R^{d\times m})$ respectively.  Similarly, we want the image $h(t,X_t)$ of $h$ to be contained in a set of matrices with $j$'th column in $L^2(\Omega, L^2(\nu_j,\mathbb R^d))$. To express the space of all such matrices in a compact manner, we introduce the following notation.
For the $\mathbb R^n$-valued measure $\nu=(\nu_1, \ldots, \nu_n)^{\mathsf T}$, we will write 
 $L^p(\nu)= L^p(\nu,\mathbb R^{d\times n})$ $(p \geq 2)$, to denote the set of measurable functions
\begin{align*}
 H:\mathbb R_0\rightarrow\mathbb R^{d\times n},
\end{align*}
such that 
\begin{align}\label{eq:LnuNorm_columnWise}
\| H\|^{p}_{L^{p}(\nu)}:=\sum_{j=1}^n\| H^{,j}\|^{p}_{L^p(\nu_j,\mathbb R^d)}<\infty.
\end{align}
Here we have used the notation  $H^{,j}$ to denote the $j$'th column of $H$. Notice also that the Bochner space  $$L^{q}(\Omega, L^{p}(\nu, \mathbb R^{d\times n})) \quad (q\geq 2)$$ consists of the measurable functions $\mathcal H:\Omega\mapsto L^{p}(\nu, \mathbb R^{d\times n})$ such that
\begin{align}
 \| \mathcal H\|^{q}_{L^{p}(\Omega, L^{p}(\nu, \mathbb R^{d\times n}))}:=E[\| \mathcal H\|^{q}_{L^p(\nu,\mathbb R^{d\times n})}]<\infty.
\end{align}
For convenience, we will sometimes omit to explicitly specify the spaces $\mathbb R^d$, $\mathbb R^{d\times m}$ and $\mathbb R^{d\times n}$, when it is clear from the context which space to consider and no confusion is possible. In this paper,  when no confusion will occur, the standard Euclidean norm $\| \cdot\|_{\R^{k\times l}}$, will be denoted by  $|\cdot|$  for any $k$, $l \in \mathbb{N}$.  

\vspace{5mm}
Hereafter we introduce the relevant spaces we work with in the sequel.
For $0\leq u \leq T$, let
\begin{align}\label{def:D}\mathcal{D}_u:=\mathcal{D}\mathcal([-r,u],\mathbb R^d),\end{align} 
denote the space of all \cadlag \ functions from $[-r,u]$ to $\mathbb{R}^d$, equipped with the uniform norm 
\begin{align}\label{def:NormD}
\|\eta\|_{\mathcal{D}_u}:=\sup_{-r\leq \theta\leq u}\{|\eta(\theta)|\}, \quad \eta \in \mathcal{D}_u.
\end{align} 
Set $\mathcal{D} := \mathcal{D}_0$.
For $2\leq p <\infty$, let $L^p_u:=L^p([-r,u], \mathbb R^d)$ and $$ M^p_u := L^p_u \times\mathbb R^d$$ with norm given by
$$\|(\eta,v) \|_{M^p_u}^p  := \| \eta  \|^p_{L^p_u}+|v|^p, \quad \eta \in M^p_u$$
Set $L^p := L_0^p$ and $M^p := M^p_0$.

\vspace{2mm}
We recall that the $M^p_u$-spaces are separable Banach spaces and  $M^2_u$ is also a Hilbert space.
On the other side $\mathcal{D}_u$ equipped with the topology given by \eqref{def:NormD} is a non-separable Banach space. The space $\mathcal{D}_u$ equipped with the Skorohod topology is separable metric space. Moreover, there exists also a topology on $\mathcal{D}_u$, equivalent to the Skorohod topology, such that $\mathcal{D}_u$ is a complete separable metric space. See e.g. \cite{MR0233396, MR0226684}. 

Observe  that if $\eta\in \mathcal{D}$, then
\begin{align}
 \| (\eta\mathbbm{1}_{[-r,0)},\eta(0)) \|^p_{M^p}= \| \eta\mathbbm{1}_{[-r,0)}\|^p_{L^p}+ |\eta(0)|^p\leq (r+1)\| \eta\|^p_{\mathcal{D}}. \label{n1}
\end{align}
By \eqref{n1}, and since the elements in $M^p$ have at most one \cadlag \ representative, the linear functional $$\eta\mapsto (\eta\mathbbm{1}_{[-r,0)},\eta(0))$$ is a linear continuous embedding of $\mathcal{D}$ into $M^p$.
Note that we will write $\| \eta\|_{M^p}$ in place of $\| (\eta\mathbbm{1}_{[-r,0)},\eta(0))\|_{M^p}$.

\vspace{5mm}
We now introduce the notion of \textit{segment} that will play an important role in this paper. 
\begin{definition}
For any stochastic process $Y:[-r,T]\times\Omega\rightarrow \mathbb R^d$, and each $t\in [0,T]$, we define the \emph{segments} 
\begin{align*}Y_t:[-r,0]\times\Omega\rightarrow \mathbb R^d,\quad\textnormal{by}\quad
 Y_t(\theta,\omega):=Y(t+\theta,\omega), \quad \theta\in[-r,0], \quad \omega\in \Omega.
\end{align*} 
In view of the arguments above, for each $t$, the segment can also be regarded as a function
$$\Omega\ni\omega\mapsto Y_t(\cdot,\omega)\in\mathcal{D}$$ 
or
$$\Omega\ni\omega\mapsto (Y_t(\cdot,\omega)\mathbbm{1}_{[-r,0)}, Y(t)) \in M^p$$
provided the necessary conditions of \cadlag \ paths or integrability.   
\end{definition}

\noindent
We recall the following definitions. 
Let $\mathcal G\subseteq \mathcal F$ be a $\sigma$-algebra on $\Omega$, containing all the $P$-null sets. 
Let $\mathcal{D}$ be equipped with the $\sigma$-algebra $\mathfrak{D}$ generated by the Skorohod topology.

\begin{definition}\label{def:DvaluedRandomVariable}
We say that a function $\eta:\Omega\rightarrow\mathcal{D}$ is  a \emph{($\mathcal G$-measurable) $\mathcal{D}$-valued random variable} if it is $\mathcal{G}$-measurable with respect to the  $\sigma$-algebra $\mathfrak{D}$, or equivalently, if the $\mathbb{R}^d$-valued function $\omega\mapsto\eta(\theta,\omega)$ is $\mathcal G$-measurable for each $\theta\in[-r,0]$. 
\end{definition}
\noindent
Recall that $ \mathfrak{D}\subsetneq\mathcal{B}(\mathcal{D}) $, where $\mathcal{B}(\mathcal{D})$ is the Borel $\sigma$-algebra generated by the topology given by the norm \eqref{def:NormD}.

\begin{definition}
We say that a function $(\eta, v) :\Omega\rightarrow M^p$ is a \emph{($\mathcal G$-measurable) $M^p$-valued random variable} if it is measurable with respect to the $\sigma$-algebras $\mathcal G$ and $\mathcal B(M^p)$, or equivalently if the function $$\omega\mapsto
\int_{-r}^0\eta(\theta,\omega)\phi(\theta)d\theta+v(\omega)\cdot u$$ is $\mathcal G$-measurable for every $(\phi,u)\in {M^p}^*=M^{\frac{p}{p-1}}$.
\end{definition}
\noindent
Notice also that if $\eta$ is a $\mathcal G$-measurable $\mathcal{D}$-valued-random variable, then it is $\mathcal G$-measurable as an $M^p$-valued random variable.
Corresponding definitions apply in the cases of the $D_u$ or $M^p_u$ spaces above.

\vspace{5mm}
We are now ready to introduce the spaces of measurable $\mathcal{D}$-valued and $M^p$-valued random variables.

Recall that $\mathcal{D}_u$ is equipped with the $\sigma$-algebra $\mathfrak{D}_u$ generated by the Skorohod topology on $\mathcal{D}_u$.
Let $\eta$ be a $\mathcal{D}_u$-valued random variable. For $p\geq 2$,  define 
$$\|\eta\|^p_{S^p(\Omega;\mathcal{D}_u)}:=E\left[\sup_{\theta\in[-r,u]}|\eta(\theta)|^p\right] = E[\|\eta(\theta)\|_{\mathcal{D}}^p]
,$$ 
and the equivalence relation $\eta_1\sim \eta_2\Leftrightarrow \|\eta_1-\eta_2\|_{S^p(\Omega;\mathcal{D}_u)}=0$. 
Let $$S^p(\Omega, \mathcal G;\mathcal{D}_u)$$ denote the space of equivalence classes of $\mathcal{D}$-valued random variables $\omega\mapsto\eta(\omega,\cdot)$ such that \\
$\|\eta\|^p_{S^p(\Omega;\mathcal{D}_u)}<\infty.$

For $p \geq 2$, let $$L^p(\Omega,\mathcal G; M^p_u ),$$ 
denote the Bochner spaces $L^p(\Omega, M^p_u)$ consisting of the   $M^p_u$-valued random variables $(\eta,v)$ such that the norm given by
\begin{align*}
 \|  (\eta,v)\|^p_{L^p(\Omega;M^p_u)}:=E[\| (\eta,v)\|^p_{M^p_u}]
\end{align*}
is finite. We recall that both $S^p(\Omega;\mathcal{D}_u)$ and $L^p(\Omega;M^p_u)$ are Banach spaces. Observe that if $\eta\in S^p(\Omega,\mathcal G;\mathcal{D})$, then
\begin{align}\label{ineq:Sp_vs_Lp1}\|(\eta,\eta(0))\|^p_{L^p(\Omega;M^p)}\leq (r+1)\| \eta\|^p_{S^p(\Omega;\mathcal{D})} \end{align}
thus, it also holds that
$$ 
S^p(\Omega,\mathcal G;\mathcal{D})\subset L^p(\Omega, \mathcal G; M^p),$$ 
and the embedding is continuous. With the appropriate boundedness and integrability conditions on a \cadlag \ adapted \mbox{process $Y$,} then for each $t$, the segment $Y_t$ can be regarded as an element in the spaces $S^p(\Omega,\mathcal F_t;\mathcal{D})$ or $L^p(\Omega,\mathcal F_t; M^p)$. 

\vspace{5mm}
In line with the definitions given above, we also use the following notation for any $u\in[0,T]$ and $2\leq p<\infty.$ 
Let 
\begin{align*}
S_{ad}^p(\Omega, \mathcal{F}_u;\mathcal{D}_u) \subseteq S^p(\Omega, \mathcal{F}_u;\mathcal{D}_u)
\end{align*} 
denote the subspace of elements in $S^p(\Omega, \mathcal{F}_u;\mathcal{D}_u)$ admitting a $\mathbb{F}$-adapted representative.
We remark that if $Z\in S^p(\Omega, \mathcal{F}_T;\mathcal{D}_T)$, then we have that 
\begin{align}\label{S^p-comparison}
 \| Z \|_{S^p(\Omega;\mathcal{D})}&\leq \| Z\|_{S^p(\Omega; \mathcal{D}_t)} \leq  \| Z\|_{S^p(\Omega; \mathcal{D}_T)}.
\end{align} 
Also, consider the Banach space
$$
L^p(\Omega; L_u^p)
$$
with the usual norm given by:
\begin{align*}
 \| Y\|^p_{L^p(\Omega; L^p_u)}:=\mathbb{E} \big [\| Y\|^p_{L^p_u}   \big] <\infty\, .
\end{align*}
Then 
$$L^p_{ad}(\Omega;L^p_u) \subseteq L^p(\Omega; L_u^p)
$$ 
denotes the subspace of elements admitting $\mathbb F$-adapted representative.

Suppose now that $Y\in L^p_{ad}(\Omega; L^p_T)$. Since  $Y(t)$ is well-defined for a.e. $t\in[-r,T]$, it makes sense to consider the segments $Y_t$ as elements in $L^p(\Omega, \mathcal F_t; L^p_t)$ for a.e. $t$. Then,
\begin{align}\label{ineq:intOfY_t}
\begin{split}
\int_{-r}^u\| (Y_t, Y(t))\|^p_{L^p(\Omega;M^p)} dt &\leq \int_{-r}^u \Big( \| Y_t\|^p_{L^p (\Omega;L^p_t)}+\| Y(t)\|^p_{L^p(\Omega; \mathbb{R}^d)}\Big)dt\\&\leq 2(r+u)\| Y_t\|^p_{L^p(\Omega; L^p_u)}.
\end{split}
\end{align}

Even though we can not consider $S^p_{ad}(\Omega; \mathcal{D}_t)$  as a subspace of $L^p_{ad}(\Omega; L^p_t)$, since the function $$S^p_{ad}(\Omega; \mathcal{D}_t)\ni \eta\mapsto\eta\in L^p_{ad}(\Omega; L^p_t)$$ is not injective, this function is continuous, and
\begin{align}\label{ineq:Sp_vs_Lp}
\| Y_t\|^p_{L^p(\Omega; L^p_t)}\leq (t+r)\| Y_t\|^p_{S^p(\Omega; \mathcal{D}_t)}.
\end{align}

\begin{remark}
In the continuous setting (see e.g. \cite{MR1652338}), the segments of an SFFDE are often considered as elements of the Bochner space $L^2(\Omega;\mathcal C)$, where $\mathcal C$ denotes the set of continuous functions from $[-r,0]$ to $\R^d$.  
We remark that the c\`{a}dl\`{a}g counterpart, namely the Bochner space $L^p(\Omega,\mathcal G;\mathcal D)$ of $\mathcal D$-valued functions turns out to be too restrictive to contain a sufficiently large class of \cadlag \ segments. This can bee seen from the following lemma:
\end{remark}

\begin{lemma}\label{lemma:continuousLevyInL2}
Suppose that $X$ is a  \cadlag \ L\'{e}vy-It\^{o} process with  $X\in L^p(\Omega,\mathcal G; \mathcal D[a,b])$. Then $X$ is continuous with probability $1$.
\end{lemma}
\noindent
To see why this holds, we first recall that by an equivalent definition of Bochner spaces (see \cite{MR1009162} for more on these spaces), $L^p(\Omega,\mathcal G;\mathcal D)$ consists of equivalence classes of the $(\mathcal G,\mathcal B(\mathcal D))$-measurable functions $X:\Omega\rightarrow\mathcal D$ such that the image $X(\Omega_0)$ is separable for some subset $\Omega_0\subset \Omega$ with $P(\Omega_0)=1$, and  $E[\| X\|^p_{\mathcal D([a,b])}] <\infty$ holds\footnote{in fact this definition is valid for $\mathcal D([a,b])$ replaced by any Banach space $V$}. By \cite[lemma 9.12]{MR3402381}, we know that $X(\Omega_0)$ is separable if and only if there exist a countable set $\mathbb T_0\in[a,b]$, such that $\Delta X(t,\omega)=0$ whenever $t\notin \mathbb T_0, \omega\in\Omega_0$. In other words, except for a negligible set of sample paths of $X$, all the jumps of $X$  occur at a countable number of times.

\begin{proof}[Proof of Lemma \ref{lemma:continuousLevyInL2}]
Since $X\in L^p(\Omega,\mathcal G; \mathcal D([a,b]))$, we can choose $\Omega_0,\mathbb T_0 $ be as above. Now since $X$ is a \cadlag \ L\'{e}vy-It\^{o} process, it also holds that $P(\omega: \Delta X(t,\omega)\neq 0 )=0$ for every $t$, and hence $$\mathcal N:= \bigcup_{t\in\mathbb{T}_0} \{\omega\in\Omega_0: \Delta X(t,\omega)\neq 0 \}$$ is a null set. But then if $\omega\in \Omega_0\setminus \mathcal N$, it holds that $\Delta X(t,\omega)=0$ for every $t$, that is $X$ is continuous  on $\Omega_0\setminus \mathcal N$ and  $P(\Omega_0\setminus \mathcal N)=1$.
\end{proof}
\color{black}
\subsection{Examples}\label{SEC:DelaT}
To illustrate some possible ways to model memory or delay in a stochastic differential equation, we include some examples of delay terms appearing in applications.

\begin{enumerate}[i)]
\item \label{example:distributedBorel} \textbf{Distributed delay}:
the functional 
\begin{equation}\label{EQN:DistDelay}
S_t\longmapsto \int_{-r}^0 S(t+\theta) \alpha(d\, \theta)\, .
\end{equation}
where $\alpha$ is a finite Borel measure on $[-r,0]$, is an example of a \emph{distributed delay}-functional.
This is a  general type of delay in the sense that examples \ref{example:AbsContDelay}), \ref{example:discrete}) below, can be regarded as particular cases of this one.
 
A general financial framework in this setting has been studied in \cite{Chang2,Chang} where the authors considered a price evolution for the stock of the form
\[
dS(t) = M(S_t) dt + N(S_t) dW(t) = \int_{-r}^0 S(t+s) \alpha_M(d s) dt +\int_{-r}^0 S(t+s) \alpha_N(d s) dW(t)\, ,
\]
$\alpha_M$ and $\alpha_N$ being suitable functions of bounded variation.
See also \cite[Sec. V]{MR1652338}, where $\alpha$ is taken as a \emph{probability measure.}
\item \label{example:AbsContDelay}\textbf{Absolutely continuous distributed delay}: in the particular case $\alpha << \mathcal{L}$, where we have denoted by $\mathcal{L}$ the Lebesgue measure, we have that the measure $\alpha$ admits a density $\kappa := \frac{d \alpha}{d \mathcal{L}} $. Therefore the functional \eqref{EQN:DistDelay} reads as
\begin{equation*}\label{EQN:AbsDelay}
S_t\longmapsto \int_{-r}^0 S(t+\theta) \kappa (\theta) d\, \theta, .
\end{equation*}

A more advanced example has been provided in \cite{MR2332878} where a functional of the form
\[
(t,S_t) \longmapsto \int_{-r}^0 \ell (t,S(t+\theta)) h(\theta) d \theta,
\]
for some functional $\ell$, has been treated.

\item \label{example:discrete} \textbf{Discrete delay}:
if we let $\alpha = \delta_\tau$,  in equation \eqref{EQN:DistDelay}, where $\delta_\tau$ is the Dirac measure concentrated at $\tau \in [-r,0]$, then we have a \textit{discrete delay} functional, namely
\begin{equation}\label{EQN:DiscDelay}
S_t \longmapsto \int_{-r}^0 S(t+\theta) \delta_\tau (d \theta) = S(t-\tau)\, .
\end{equation}

A discrete delay model using functionals on the form \eqref{EQN:DiscDelay}, is widely used in concrete applications, spanning from mathematical biology, as in the case of the delayed Lotka-Volterra model, see, e.g. \cite{diekmann1995delay, MR1218880,MR1652338}, to mathematical finance, as it happens for the delayed Black-Scholes model, see, e.g. \cite{SMohammed,MR2241374}, or for commodities markets, see, e.g., \cite{kuchler}. In particular, in \cite{SMohammed}, the authors give an explicit form for the price a European call option written on an underlying evolving as
\[
dS(t) = \mu S(t-a) dt + \sigma(S(t-b)) dW(t)\, ,
\]
for $\mu \in \R$ and a suitable function $\sigma$.

A particular case of the discrete delay example is the no delay case, i.e. $\tau=\delta_0$. A multiple delay case, can be defined by letting  $\alpha = \sum_{i=1}^N \delta_{\tau_i}$, $\tau_i \in [-r,0]$,  $i=1,2,\dots,N$.
\item \textbf{Brownian delay}: our setting allows also to consider delays with respect to a Brownian motion, namely
\[
S_t \longmapsto \int_{t-r}^t S(\theta) d\, W(\theta)\, .\\
\]
Hence this permits to take {\it noisy memory} models into account. These cases are arising e.g. in the modelling of stochastic volatility see, e.g. \cite{MR2241374, swishchuk2013modeling} and when dealing with stochastic control problems, see e.g. \cite{dahl2014optimal}.
\item \textbf{L\'{e}vy delay}: similarly to the Brownian delay, we can also 
consider a delay with respect to a square integrable L\'{e}vy process of the form
\[
S_t \longmapsto\int_{t-r}^t S(\theta) d\, L(\theta)\, .\\
\]

Such type of delay has been employed in \cite{swishchuk2013modeling} in order to consider some stochastic volatility models related to energy markets.
\item \label{example:meanfield} \textbf{Mean field delay}: we can consider a delay of the form
\[
S_t\longmapsto\mathbb{E}\left [\int_{-r}^0 S(t+\theta) \,  \alpha(d\theta) \right]\, ,
\]
where $\alpha$ is as in example \ref{example:distributedBorel}), see e.g. \cite{agram2014optimal}.
\end{enumerate}

\color{black}\label{TheEquationsExistenceUniquenessMomentestimates}
\subsection{ $\mathcal{D}$ framework}\label{sec:functionalsOnS}

Fix $p \in [2, \infty)$.
Consider again the equation
\begin{align}\label{eq:Main12}
 d X(t)&=f(t, X_t) d t+g(t,X_t)d W(t)+ \int_{\mathbb{R}_0}h(t,X_t)(z)\tilde N(d t, d z)\\
X_0&=\eta, \notag
\end{align}

In this section, we require that $f(t,\cdot)$, $g(t,\cdot)$, $h(t,\cdot)$ are defined on  $S^p(\Omega,\mathcal F_t;\mathcal{D})$ for each fixed $t$.
Therefore, we introduce the space
\begin{align}
 \mathbf S^{\mathbb F}_p:=\{(t,\psi)\in [0,T]\times  S^p(\Omega,\mathcal F;\mathcal{D})\textnormal{ such that }  \psi\in S^p(\Omega,\mathcal F_t;\mathcal{D})  \},
\end{align}
as the domain of the coefficient functionals $f,g,h$ in the SFDDE \eqref{eq:Main12}.
In particular, we will require that:
\begin{align*}
    f&:\mathbf S^{\mathbb F}_p\to L^p(\Omega, \mathbb{R}^d)\\
    g&:\mathbf S^{\mathbb F}_p\to L^p(\Omega, \mathbb{R}^{d\times m}),\\
    h&:\mathbf S^{\mathbb F}_p\to L^p(\Omega, L^2(\nu,\mathbb R^{d\times n})).
\end{align*}
Moreover,
$$\eta\in S^p(\Omega,\mathcal F_0;\mathcal{D})$$

\noindent
To ensure that the integrals are well-defined,  the following assumptions are imposed on the coefficient functionals $f,g$ and $h$.

\begin{assumption}[$\mathcal P$]
Whenever $Y \in S^p_{ad}(\Omega; \mathcal{D}_T),$ the process
\begin{align}
[0,T]\times\Omega\times\mathbb R_0\,\ni \,(t,\omega,z)\mapsto h(t,Y_t)(\omega)(z)\,\, \in \mathbb R^{d\times n}
\end{align}
has a predictable version, and 
\begin{align*}
[0,T]\times\Omega \, \ni \,(t,\omega)&\mapsto f(t,Y_t)(\omega)\,\, \in \mathbb R^{d},\quad\\
[0,T]\times\Omega \, \ni \,(t,\omega)&\mapsto g(t,Y_t)(\omega)\,\, \in \mathbb R^{d}
\end{align*}
have progressively measurable versions.
\end{assumption}
Predictable and progressive should be interpreted in the standard sense for $\R^{k}$-valued processes (see e.g. \cite{MR2512800}). We emphasise that the integrals in \eqref{eq:Main12} should be interpreted with respect to the predictable and progressive versions of the respective integrands. For a range of SFDE’s likely to be encountered in applications, the assumption $\mathcal P$ is fairly easy to verify.

\begin{example}
Most of the examples presented in Section \ref{SEC:DelaT} satisfy Assumption $\mathcal{P}$. For instance, the functional displayed in \ref{example:distributedBorel}), which is more general than \ref{example:AbsContDelay}), \ref{example:discrete}), is predictable whenever the point zero is not an atom of the measure $\alpha$, i.e. the discrete delay in \eqref{EQN:DiscDelay} is not allowed when $\tau=0$. The mean-field delay in example \ref{example:meanfield}) is deterministic and hence predictable. The Brownian delay can also be considered, since the process $t\mapsto \int_{t-r}^{t}S(\theta)dW(\theta)$ is a continuous martingale, in particular it admits a version with left-limits. 
\end{example}

\begin{definition}\label{DEF:SolS}\color{black} Suppose that the assumption $\mathcal P$ holds.\color{black}
We say that $X\in S^p_{ad}(\Omega; \mathcal{D}_T)$ is a 
\emph{strong solution} to the equation  \eqref{eq:Main12} if for each $t\in[0,T]$
\small
\begin{align}\label{eq:SFDEgeneral2}
 X(t)&=\eta(0)+\int_0^t f(s, X_s)ds +\int_0^t g(s,X_s)d W(s)+\int_0^t \int_{\mathbb{R}_0}h(s,X_s)(z)\tilde N(d s,d z)\\      
X_0&
=\eta.\notag
\end{align}\normalsize
If the solution is unique, we will write $^\eta X$ to denote the solution of \eqref{eq:SFDEgeneral2} with initial datum $^\eta X_0=\eta$.
\end{definition}
To prove existence and uniqueness of the solution of the SFDDE, we  rely on the following result.

\begin{lemma}[Kunita's inequality]\label{lem:Kunita}
Let $q\geq 2$.
Suppose  that $F, G$ and $H$ are predictable processes taking values in $\mathbb R^d$, $\mathbb R^{d\times m}$ and $\mathbb R^{d\times n}$ respectively. If $$Y(t)=Y_0+\int_0^t F(s)ds+\int_0^tG(s)dW(s)+\int_0^t\int_{\mathbb{R}_0}H(s,z)\tilde N(ds,dz),\quad t\in [0,T],$$
then there exists a constant $C=C(q,d,m,n,T)$, independent of the processes $F,G$ and $H$ and the initial value $Y_0$, such that whenever $t\leq T$ the following inequality holds
\begin{align}\label{ineq:KunitaNorm}
\begin{split} 
\mathbb{E}[&\sup_{0\leq u\leq t}|Y(t)|^q]\leq C\Big\{\| Y_0\|_{L^q(\Omega, \mathbb R^d)}^q+ \int_0^t \Big(\| F(s) \|_{L^q(\Omega,\mathbb R^d)}^q+\| G(s) \|_{L^q(\Omega, \mathbb R^{d\times m})}^q\\
&+\| H(s) \|^q_{L^q(\Omega, L^q(\nu))} +\| H(s) \|^q_{L^q(\Omega, L^2(\nu))}\Big)ds\Big\}
\end{split}
\end{align}
\end{lemma}
For $n=1$ (and arbitrary $m$ and $d$), this is a rewritten version of Corollary 2.12 in \cite{MR2090755}. We have justified the extension to general $n$ in Appendix \ref{sec:Kunita}.

\subsubsection{Existence, uniqueness and moment estimates}
Before giving sufficient conditions for existence and uniqueness of solutions to the equation \eqref{eq:Main12}, we will establish a set of hypotheses.
\begin{assumption}\label{HYP:D1}
\begin{description}
\item[$(\mathbf{D_1})$]\label{hyp:Lip} There exists $L>0$,  such that whenever $t\in[0,T]$ and $\eta_1,\eta_2\in S^p(\Omega,\mathcal F_t;\mathcal{D})$, then
      \begin{align*}
       \| f&(t,\eta_1)-f(t,\eta_2)\|^p_{L^p(\Omega; \mathbb{R}^d)}+\| g(t,\eta_1)-g(t,\eta_2)\|^p_{L^p(\Omega;\mathbb{R}^{d\times n})}\\
	&+\| h(t,\eta_1)-h(t,\eta_2)\|^p_{L^p(\Omega,L^p(\nu))}
	+\| h(t,\eta_1)-h(t,\eta_2)\|^p_{L^p(\Omega,L^2(\nu))}\\
	&\leq  L \| \eta_1-\eta_2 \|^p_{S^p(\Omega;\mathcal{D})}.
      \end{align*}
\item[$(\mathbf{D_2})$]\label{hyp:linGrowth} There exists $K>0$, such that whenever $t\in[0,T]$ and $\eta\in S^p(\Omega,\mathcal F_t;\mathcal{D})$, then
\begin{align}\label{condition:linearGrowth1}
\begin{split}
 \| f&(t,\eta)\|^p_{L^p(\Omega; \mathbb{R}^d)}+\| g(t,\eta)\|^p_{L^p(\Omega; \mathbb{R}^{d\times n})}\\
&+ \| h(t,\eta)\|^p_{L^p(\Omega,L^p(\nu) )}+ \| h(t,\eta)\|^p_{L^p(\Omega,L^2(\nu))}\\
& \leq K(1+\| \eta \|^p_{S^p(\Omega;\mathcal{D})}).\end{split}
\end{align}
\end{description}
\end{assumption}
\begin{remark}
As usual, $\mathbf D_2$ is implied by $\mathbf D_1$, if we assume that whenever $\eta=0$, the left-hand-side of inequality \eqref{condition:linearGrowth1} is bounded by some $K'$,  uniformly in $t\in[0,T]$.
\end{remark}

\begin{theorem}[\textbf{Existence and Uniqueness I}]\label{thm:existenceAndUniquenessEp}$ $
\color{black}Consider equation \eqref{eq:Main12} with $\mathcal{P}$ satisfied. \color{black}
\begin{description}
 \item[(i)] Suppose that \color{black} assumption \color{black} $\mathbf{D_1}$ holds. If $X$, $Y\in S_{ad}^p(\Omega;\mathcal{D}_T)$ are strong solutions to \eqref{eq:Main1}, then $X=Y$.
\item[(ii)] Suppose that \color{black} assumptions \color{black} $\mathbf{D_1}$ and $\mathbf{D_2}$ hold. Then there exists a strong solution $X\in S^p_{ad}(\Omega; \mathcal{D}_T)$ to the equation \eqref{eq:Main1}. Moreover, there exists  $D=D(K,p,T,d,m,n)>0$,
such that 
\begin{align}\label{ineq:LpEstimate}
\| X\|^p_{S^p_{ad}(\Omega; \mathcal{D}_T)}
 \leq e^{Dt}(Dt+\| \eta\|^p_{S^p(\Omega;\mathcal{D})})\, 
\end{align}
whenever $t\leq T$.
\end{description}
\end{theorem}

\begin{proof}
We will use a standard Picard iteration argument to show that a solution exists. 
First, we define, for each $k\geq 0$, a sequence of processes in $S^p_{ad}(\Omega; \mathcal{D}_T)$ inductively by
\begin{align*}
 X^1(t)&=\eta(0),&t\in[0,T],\\
X^1_0&= \eta \notag\\
\begin{split}X^{k+1}(t)&=\eta(0)+\int_0^t f(s,X^k_s) ds+ \int_0^t g(s,X^k_s)d B(s)\\
&\quad+\int_0^t h(s,X^k_s)(z)\tilde N(d s,d z),
\end{split}&t\in[0,T]\\ 
X^{k+1}_t&=\eta.\notag
\end{align*}
We immediately have that $X^1\in S^p_{ad}(\Omega; \mathcal{D}_T)$. Also if we assume that $X^k\in S^p_{ad}(\Omega; \mathcal{D}_T)$,
then by assumption $f(X^k)$, $g(X^k)$, and  $h(X^k)$ admit progressive and predictable versions respectively. Thus by assumption ($\mathbf{D_2}$) it follows that
\begin{align}\label{ineq:existence_proof1}
\begin{split}
\int_0^T &\Big(\| f(t, X^k_t)\|^p_{L^p(\Omega;\mathbb{R}^d)} +  \| g(t, X^k_t)\|^p_{L^p(\Omega;\mathbb{R}^{d\times n})} \\
&+\| h(t, X^k_t)\|^p_{L^p(\Omega, L^p(\nu))}+ \| h(t, X_t)\|^p_{L^p(\Omega,L^2(\nu))}\Big) dt\\
&\leq \int_0^T K(1+\| X^k_t \|^p_{S^p(\Omega;\mathcal{D})}) dt 
\leq KT(1+\| X^k\|^p_{S^p_{ad}(\Omega; \mathcal{D}_T)})
<\infty.
\end{split}
\end{align}
In particular, the integrands of $X^{k+1}$ are It\^o integrable, so that $X^{k+1}$ is \cadlag \ and adapted, and finally by Kunita's inequality, we have that $X^{k+1}\in S^p_{ad}(\Omega; \mathcal{D}_T)$. 
 
We now claim that for each $k\in\mathbb{N}$ the following estimate holds for every $t\in[0,T]$,
\small\begin{align}\label{est:XCauchy}
\| X^{k+1}-X^{k}\|^p_{S^p_{ad}(\Omega; \mathcal{D}_t)}\leq \frac{(LCt)^{k-1}}{(k-1)!} \| X^{2}-X^{1}\|^p_{S^p_{ad}(\Omega; \mathcal{D}_T)}.
\end{align}\normalsize
This trivially holds when $k=1$.  Now suppose that \eqref{est:XCauchy} holds for each $t\in[0,T]$. Using the definition of $X^{k+2},\, X^{k+1}$,
 Kunita's inequality \eqref{ineq:KunitaNorm}, and assumption ($\mathbf{D_2}$), we find that
\small
\begin{align*}
\| X^{k+2}-X^{k+1}\|^p_{S^p_{ad}(\Omega; \mathcal{D}_t)}
&\leq C\int_0^t\Big(\| f(s,X_s^{k+1})-f(s,X_s^k) \|_{L^p(\Omega;\mathbb{R}^d)}^p+\| g(s,X_s^{k+1})-g(s,X_s^k) \|_{L^p(\Omega;\mathbb{R}^{d\times n})}^p\\
& + \| h(s,X_s^{k+1})-h(s,X_s^k)\|_{L^p(\Omega, L^p(\nu))}^p + \| h(s,X_s^{k+1})-h(s,X_s^k)\|_{L^p(\Omega,L^2(\nu))}^p \Big)ds\\
&\leq LC \int_0^t\| X_s^{k+1}- X_s^{k}\|^p_{S^p(\Omega;\mathcal{D})}ds\leq LC \int_0^t\| X^{k+1}- X^{k}\|^p_{S^p(\Omega; \mathcal{D}_s) } ds\\
&\leq LC\int_{0}^t \frac{(LCs)^{k-1}}{(k-1)!}\| X^{2}- X^{1}\|^p_{S^p_{ad}(\Omega; \mathcal{D}_T)}ds= \frac{(LCt)^{k}}{k!}\| X^{2}- X^{1}\|^p_{S^p_{ad}(\Omega; \mathcal{D}_T)}.
\end{align*}\normalsize
Now, by induction, \eqref{est:XCauchy} holds for each $k\in \mathbb{N}$. In particular
\small\begin{align*}
\| X^k-X^i \|^p_{S^p_{ad}(\Omega; \mathcal{D}_t)}\leq \| X^2-X^1 \|^p_{S^p_{ad}(\Omega; \mathcal{D}_T)}\sum_{j=\min\{k,i\}}^{\infty} \frac{(LTC)^{j-1}}{(j-1)!}\to 0  \, ,\quad \mbox{ as } k , \, i \to \infty \, ,
\end{align*}\normalsize
so that $\{X^k\}_{k\geq 0}$ is a Cauchy sequence in $S^p_{ad}(\Omega; \mathcal{D}_T)$.
Since $S^p_{ad}(\Omega; \mathcal{D}_T)$ is complete, we have that $\{X^k\}_{k\geq 0}$ converges to some $X$ in $S^p_{ad}(\Omega; \mathcal{D}_T)$. Clearly $X_0=\eta$ $P$-a.s. 

We will now show that the limit $X$ satisfies \eqref{eq:SFDEgeneral2} by showing that
\small\begin{align}\label{eq:XminusX}
d:=\E\Big[\sup_{0\leq t\leq T}&\Big|X(t)-\Big\{\eta(0)+\int_0^{t}f(s,X_s)ds+ \int_0^{t}g(s,X_s)dW(s)\nonumber\\
& + \int_0^{t}\int_{\mathbb{R}_0}h(s,X_s)(z)\tilde N(ds,dz)\Big\}\Big|^p\Big]^{1/p}= 0
\end{align}\normalsize
For arbitrary $k$, we subtract $X^{k+1}$ and add its integral representation inside the supremum in \eqref{eq:XminusX}. Then by the triangle inequality, Kunita's inequality, and finally the Lipschitz condition $(\mathbf{D_1})$ we find that
\small\begin{align*}
d
&\leq \| X-X^{k+1}\|_{S^p_{ad}(\Omega; \mathcal{D}_T)} +  \Big\{C\int_0^T\Big(\| f(t,X_t^k)-f(t,X_t) \|_{L^p(\Omega;\mathbb{R}^d)}^p\\
 &+\| g(t,X_t)-g(t,X_t^k) \|_{L^p(\Omega;\mathbb{R}^{d\times n})}^p + \| h(t,X_t)-h(t,X_t^k)\|_{L^p(\Omega, L^p(\nu))}^p\\
&\quad + \| h(t,X_t)-h(t,X_t^k)\|_{L^p(\Omega,L^2(\nu))}^p \Big)dt\Big\}^{1/p}\\
&\leq \| X-X^{k+1}\|_{S^p_{ad}(\Omega; \mathcal{D}_T)} + \big\{CL \int_0^T\| X_t-X^k_t\|^p_{S^p(\Omega;\mathcal{D})}dt\big\}^{1/p}\\
&\leq  \| X-X^{k+1}\|_{S^p_{ad}(\Omega; \mathcal{D}_T)} + (CLT)^{1/p}\| X-X^k\|_{S^p_{ad}(\Omega; \mathcal{D}_T)}\rightarrow 0.
\end{align*}\normalsize
Since for any $\epsilon >0$ we have that $0\leq d < \epsilon$, it follows that $d=0$, and hence a solution exists.
 
Suppose now that $X$ and $Y$ are solutions of \eqref{eq:Main12}. We will show that $X=Y$. 
Exploiting the integral representation of $X$ and $Y$, Kunita's inequality and the Lipschitz condition $(\mathbf{D_1})$, we have that, for all $t\in [0,T]$,
\small\begin{align*}
\| X-Y\|^p_{S^p_{ad}(\Omega; \mathcal{D}_t)}\leq& C\int_0^t\Big(\| f(s,X_s)-f(s,Y_s)\|^p_{L^p(\Omega;\mathbb{R}^d)}+\| g(s,X_s)-g(s,Y_s)\|^p_{L^p(\Omega;\mathbb{R}^{d\times n})}\\
&+\| h(s,X_s)-h(s,Y_s)\|^p_{L^p(\Omega, L^p(\nu))}+\| h(s,X_s)-h(s,Y_s)\|^p_{L^p(\Omega,L^2(\nu))} \Big)ds\\
\leq & CL\int_0^t \| X_s-Y_s\|^p_{S^p(\Omega;\mathcal{D})} ds\leq  CL\int_0^t \| X-Y\|^p_{S^p(\Omega; \mathcal{D}_s) }ds.
\end{align*}\normalsize
and thus we have $\| X-Y\|^p_{S^p_{ad}(\Omega; \mathcal{D}_t)}=0$ for every $t\in[0,T]$ from Gr\"{o}nwall's inequality.

Similarly, if $X$ is a solution to \eqref{eq:Main12}, 
from the integral representations, Kunita's inequality and the linear growth condition $(\mathbf{D_2})$ we have that
\small\begin{align*}
\| X\|^p_{S^p_{ad}(\Omega; \mathcal{D}_t)}\leq &C\Big\{\|\eta\|^p_{S^p(\Omega;\mathcal{D})}
+\int_0^t \Big(\| f(s,X_s)\|^p_{L^p(\Omega;\mathbb{R}^d)}\\
&\quad+ \| g(s,X_s)\|^p_{L^p(\Omega;\mathbb{R}^{d\times n})}+  \| h(s,X_s)\|^p_{L^p(\Omega, L^p(\nu))} + \| h(s,X_s)\|^p_{L^p(\Omega,L^2(\nu))}\Big)ds\Big\}\\
&\leq C\Big(\|\eta\|^p_{S^p(\Omega;\mathcal{D})}+ K\big(\int_0^t1+\| X_s \|^p_{S^p(\Omega;\mathcal{D})}ds \big)\Big)\leq C\|\eta\|^p_{S^p(\Omega;\mathcal{D})}+ CKt+ CK \int_{0}^t \| X\|^p_{S^p(\Omega; \mathcal{D}_s) } ds \, ,
\end{align*}\normalsize
so applying Gr\"{o}nwall's inequality we obtain
\small
\begin{align*}
\| X\|^p_{S^p_{ad}(\Omega; \mathcal{D}_t)}\leq \Big(C\|\eta\|^p_{S^p(\Omega;\mathcal{D})}+ CKt\Big)e^{CKt},
\end{align*}\end{proof}\normalsize
for all $t\in [0,T]$.

\begin{remark}[\textbf{Path dependent SDEs}]\label{REM:DetF}
Suppose that for each $t\in[0,T]$ and every  $\eta\in S_{ad}^p(\Omega,\mathcal{F}_0;\mathcal{D})$ it holds that,
\begin{align*}f(t,\eta)(\omega)&=F(t,\eta(\omega)) \\
g(t,\eta)(\omega)&=G(t,\eta(\omega))\\
h(t,\eta,\omega,\zeta)&=H(t,\eta(\omega),\zeta),
\end{align*}
 $P$-a.s for some deterministic functionals
\begin{align*}F:&[0,T]\times \mathcal{D}\rightarrow \mathbb{R}^{d},\\
G:&[0,T]\times \mathcal{D}\rightarrow \mathbb{R}^{d\times m},\\
H:&[0,T]\times\mathcal{D}\rightarrow  L^2(\nu)\cap L^p(\nu).
\end{align*}
then the assumptions $(\mathbf{D_1})$ and $(\mathbf{D_2})$ hold whenever $F,G$  are Lipschitz continuous in the second variable, uniformly with respect to the first, and $H$ is Lipschitz continuous in the second variable, uniformly with respect to the first, using both norms $\| \cdot\|_{L^2(\nu)}$ and $\| \cdot\|_{L^p(\nu)}$.
\end{remark}

\subsection{$M^p$ framework}\label{sec:M2existenceAndUniqueness}
Now, consider equation
\begin{align}\label{eq:Main13}
\begin{split} d X(t)&=f(t, X_t,X(t)) d t+g(t,X_t,X(t))d W(t)+ \int_{\mathbb{R}_0}h(t,X_t,X(t))(z)\tilde N(d t, d z)\\
(X_0,X(0))&=(\eta,x). 
\end{split}
\end{align}
Here \eqref{eq:Main13} we have used the notation $f(\cdot,X_t,X(t))$ to emphasize the structure of the product space of $M^p$. Now for each $t\in [0,T]$ we will require that $(X_t,X(t))$ belongs to the space $L^p(\Omega,\mathcal F_t;M^p)$ for some $p\in [2,\infty)$, that will be fixed throughout the section. Therefore, we introduce
\begin{align}
 \mathbf L^{\mathbb F}_p:=\{(t,(\psi,v))\in [0,T]\times  L^p(\Omega,\mathcal F;M^p)\textnormal{ such that }  (\psi,v)\in L^p(\Omega,\mathcal F_t;M^p)  \},
\end{align}
In particular, we will require that:
\begin{align*}
f&:\mathbf L^{\mathbb F}_p\rightarrow L^p(\Omega, \mathbb{R}^d))\\
g&:\mathbf L^{\mathbb F}_p\rightarrow L^p(\Omega, \mathbb{R}^{d\times m}),\\
h&:\mathbf L^{\mathbb F}_p\rightarrow L^p(\Omega, L^2(\nu,\mathbb R^{d\times n})).
\end{align*}
Moreover,
$$(\eta,x)\in L^p(\Omega,\mathcal F_0;M^p).$$

To ensure that the integrals are well-defined,  the following assumptions are imposed on the coefficient functionals $f,g$ and $h$.
\begin{assumption}[$\mathcal Q$]
For $Y \in L^p_{ad}(\Omega; L_T^p)$, the process \begin{align*}
[0,T]\times\Omega\times\mathbb R_0\,\ni,(t,\omega,z)\mapsto h(t,Y_t,Y(t))(\omega)(z)\,\, \in \mathbb R^{d\times n}
\end{align*}
has a predictable version, and 
\begin{align*}
[0,T]\times\Omega,\ni(t,\omega)&\mapsto f(t,Y_t,Y(t))(\omega)\,\, \in \mathbb R^{d},\quad\\
[0,T]\times\Omega,\ni(t,\omega)&\mapsto g(t,Y_t,Y(t))(\omega)\,\, \in \mathbb R^{d\times m}
\end{align*}
have progressively measurable versions.
\end{assumption}

\begin{definition}
 We say that $X\in L^p_{ad}(\Omega; L^p_T)$ is a 
\emph{strong solution} to \eqref{eq:Main13} if for each $t\in[0,T]$
\begin{align}\label{eq:SFDEgeneral3}
\begin{split}
 X(t)&
=x+\int_0^t f(s, X_s,X(s))ds +\int_0^t g(s,X_s,X(s))d W(s)+\int_0^t \int_{\mathbb{R}_0}h(s,X_s,X(s))(z)\tilde N(d s,d z)\\      
(X_0,X(0))&
=(\eta,x).
\end{split}
\end{align}
If the solution is unique, we will sometimes write $^{\eta,x} X$ to denote the solution of \eqref{eq:SFDEgeneral3} with initial data $(X_0, X(0))=(\eta,x).$
\end{definition}
\noindent

\begin{prop}\label{p: cadlag}
Let $Y:[0,T]\times \Omega\rightarrow\R^d$ be a  stochastic process with a.s. c\`{a}dl\`{a}g sample paths. Then the associated $M^p$-valued \emph{segment process}
\begin{align}\label{eq:Mp-ValuedCadlag}
 [0,T]\times\Omega\ni (t,\omega)\mapsto (Y_t(\omega), Y(t, \omega))\in M^p
\end{align}
is a.s. \cadlag.
\end{prop}
\noindent
Observe that the property that the segment process is \cadlag \ whenever $Y$ is \cadlag,   depends on the topology of the infinite dimensional space $M^p$. In general, such property does not hold if we replace $M^p$ with $\mathcal{D}$.
\begin{proof}[Proof of Proposition \ref{p: cadlag}]
It suffices to show that if $Y(\omega):[-r,T]\rightarrow \mathbb R^d$ is a \cadlag \ path, then the function 
 \begin{align}\label{eq:Mp-ValuedCadlag}
 [0,T]\ni t\mapsto (Y_t(\omega), Y(t, \omega)) \in M^p
\end{align} is also \cadlag.  The function \eqref{eq:Mp-ValuedCadlag}  is right continuous. In fact, for every sequence $r_k, k\in \mathbb N $ with $r_k >0$ and $r_k\rightarrow 0$ as $k\rightarrow\infty$, we have that 
\begin{align*}
\lim_{k\to \infty} &\|Y_{t+r_k}(\omega)-Y_t(\omega), Y(t+r_k,\omega)-Y(t,\omega))\|_{M^p}^p\\
&= \lim_{k\to \infty} \int_{-r}^0 |Y(t+r_k+\beta, \omega)-Y(t+\beta, \omega)|^p d\beta + \lim_{k\to \infty} |Y(t+r_k, \omega)-Y(t, \omega)|^p=0.
\end{align*}
by the dominated convergence theorem.
Now given $t\in [0,T]$, we define $(Y^-_t(\omega), Y^-(t, \omega))\in M^p$ by
\begin{align*}
Y^-_t(\theta,\omega)= 
\begin{cases}
 Y_t(\theta,\omega), & \theta\in [-r,0)\\
 \lim_{u\rightarrow 0^-} Y_t(u,\omega),& \theta=0.
\end{cases}
\end{align*}
Consider $r_k$ as above, we can use the dominated convergence theorem to observe that 
\begin{align*}
\lim_{k\to \infty}& \|(Y_{t-r_k}(\omega)-Y^-_t(\omega), Y(t-r_k,\omega)-Y^-(t,\omega))\|_{M^p}^p 
\\&= \lim_{k\to \infty} \int_{-r}^0 |Y(t-r_k+\beta, \omega)-Y^-(t+\beta, \omega)|^p d\beta + \lim_{k\to \infty} |Y(t-r_k, \omega)-Y^-(t, \omega)|^p=0,
\end{align*}
and hence the function \eqref{eq:Mp-ValuedCadlag} has left limits.
\end{proof}

\subsubsection{Existence and uniqueness}

The $L^p(\Omega;M^p)$-analogue of the hypotheses $(\mathbf D_1)$ and $(\mathbf D_2)$, are defined below.
\begin{assumption}\label{HYP:Lip}
\begin{description}
\item[$(\mathbf{L_1})$]\label{hyp:Lip}There exists $L>0$,  such that whenever $t\in[0,T]$ and \mbox{$(\eta_1,x_1), (\eta_2,x_2)\in L^p(\Omega,\mathcal F_t;M^p)$,} then
      \begin{align*}
       \| f&(t,\eta_1,x_1)-f(t,\eta_2,x_2)\|^p_{L^p(\Omega;\mathbb{R}^d)}+\| g(t,\eta_1,x_1)-g(t,\eta_2,x_2)\|^p_{L^p(\Omega;\mathbb{R}^{d\times n})}\\
	&+\| h(t,\eta_1,x_1)-h(t,\eta_2,x_2)\|^p_{L^p(\Omega,L^p(\nu))}+\| h(t,\eta_1,x_1)-h(t,\eta_2,x_2)\|^p_{L^p(\Omega;L^2(\nu))}\\
	&\leq  L \| (\eta_1,x_1)-(\eta_2,x_2) \|^p_{L^p(\Omega;M^p)}.
      \end{align*}
\item[$(\mathbf{L_2})$] There exists $K>0$, such that whenever 
$t\in[0,T]$ and $(\eta,v)\in L^p(\Omega,\mathcal F_t;M^p)$, then
\begin{align*}
 \| f&(t,\eta,x)\|^p_{L^p(\Omega;\mathbb{R}^d)}+\| g(t,\eta,x)\|^p_{L^p(\Omega;\mathbb{R}^{d\times n})}\\
&+ \| h(t,\eta,x)\|^p_{L^p(\Omega,L^p(\nu) )}+ \| h(t,\eta,x)\|^p_{L^p(\Omega,L^2(\nu))}\\
& \leq K(1+\| (\eta,x) \|^p_{L^p(\Omega;M^p)}).\nonumber
\end{align*}
\end{description}
\end{assumption}

\begin{theorem}[\textbf{Existence and Uniqueness II}]\label{thm:existenceAndUniqueness_II}$ $
\color{black}Consider \eqref{eq:Main13} with $\mathcal{Q}$ satisfied.\color{black}
\begin{description}
\item[(i)] Let $(\eta,x)\in L^p(\Omega,\mathcal{F}_0;M^p)$ such that $\eta$ is c\`{a}dl\`{a}g $P$-a.s. and $X\in L^p_{ad}(\color{black}\Omega; L^p_T\color{black})$ be a strong solution to equation  \eqref{eq:Main13}.
Then the segment process
\begin{align}
\Omega\times[0,T]\ni(t,\omega)\mapsto (X_t(\omega), X(t,\omega))\in M^p
\end{align} has a  \cadlag \ modification.
 \item[(ii)] Suppose that \color{black}assumption \color{black} $(\mathbf{L_1})$ holds. If $X,Y\in L^p_{ad}(\color{black}\Omega; L^p_T\color{black})$ are strong solutions to \eqref{eq:Main13}, then $X=Y$.
\item[(iii)] Suppose that \color{black} assumptions \color{black} $(\mathbf{L_1})$ and $(\mathbf{L_2})$ hold. Then there exists a strong solution $X$ to equation \eqref{eq:Main13}. Moreover, there exists  $D=D(K,p,T,d,m,n)>0$,
such that
\begin{align}\label{ineq:LpEstimate}
\| X\|^p_{L^p_{ad}(\color{black}\Omega; L^p_t\color{black})}
 \leq e^{Dt}(Dt+\| (\eta,x)\|^p_{L^p(\Omega;M^p)}),
\end{align}whenever $t\leq T$.
\end{description}
\end{theorem}
\color{black}
\begin{proof}
\textbf{(i)}
Recall that since $X$ is a strong solution of \eqref{eq:Main13}, it is a semimartingale on $[0,T]$ and hence it admits a modification which is \cadlag \ on $[0,T]$. Since $X_0=\eta$ is \cadlag,  $X$ is \cadlag, on $[-r,T]$. By Proposition \ref{p: cadlag} (i) holds.
\color{black}

\textbf{(ii,iii)} The proof is based on the same argument as for the proof of Theorem \ref{thm:existenceAndUniquenessEp}. For the sake of brevity we do not write out the details. However, we remark that if one replaces the norms $\parallel \cdot\parallel_{S^p(\Omega; \mathcal{D})}$ and $\parallel \cdot\parallel_{S^p_{ad}(\Omega; \mathcal{D}_t)}$, with the norms $\parallel \cdot\parallel_{L^p(\Omega; M^p)}$ and $\parallel \cdot\parallel_{L^p_{ad}(\Omega; L^p_t))}$ respectively, then all the inequalities hold true, except for the choice of constants.
As an example, we provide the following $M^p$ analogue of \eqref{ineq:existence_proof1}, namely
\begin{align}
\begin{split}
\int_0^T & \Big(\| f(t, X^k_t,X^k(t))\|^p_{L^p(\Omega;\mathbb{R}^d)} +  \| g(t, X^k_t,X^k(t))\|^p_{L^p(\Omega;\mathbb{R}^{d\times n})} \\
&+\| h(t, X^k_t,X^k(t))\|^p_{L^p(\Omega, L^p(\nu))}+ \| h(t, X_t^k,X^k(t))\|^p_{L^p(\Omega,L^2(\nu))}\Big) dt\\
&\leq \int_0^T K(1+\| (X^k_t, X^k(t)) \|^p_{L^p(\Omega;M^p)}) dt 
\leq KT(1+(T+R)\| X^k\|^p_{L^p_{ad}(\Omega; L^p_t)})
<\infty.
\end{split}
\end{align}
This follows immediately by the assumption $\mathbf L_2$ and inequality \eqref{ineq:intOfY_t}
\end{proof}

\color{black}

Let us stress that when the initial value is \cadlag, then the setting of Section \eqref{sec:functionalsOnS} is more general than the one in this section. In fact, the assumptions $(\mathbf{L_1})$ and $(\mathbf{L_2})$ imply assumptions $(\mathbf{D_1})$ and $(\mathbf{D_2})$, respectively.

\subsection{Robustness SFDDEs}\label{robust}

In the present section we study robustness of SFDDE to changes of the noise. In particular, we want to approximate the solution of an SFDDE $X$,  with an approximate processes $X^{\epsilon}$, where  $X^{\epsilon}$ are defined by substituting the integrals with respect to the small jumps with  integrals with respect to scaled Brownian motions.  We follow rather closely, the presentation in \cite{MR2814479} for ordinary SDE's and remark that a related problem is also considered in \cite{MR2784742}. In this paper we also include a new ingredient, by giving sufficient conditions which ensure that the approximations $X^{\epsilon}$ converge to $X$ in the $p$'th mean.

The main motivation for studying such robustness problem is that it is difficult to perform simulations of distributions corresponding to a L\'{e}vy process. Indeed, simulation of such distributions are often performed by neglecting the jumps below a certain size $\epsilon$. However, when needed to preserve the variation of the infinite activity L\'evy process, a scaled Brownian motion is typically replacing the small jumps.  Under some additional assumptions, it is known that given a square integrable (1-dimensional) L\'evy process with L\'evy measure $\mu$ and compensated Poisson random measure $\widetilde{M}$,  the expression
\begin{align*}
\int_{|z|<\epsilon} z^2\mu(dz)^{-1/2}\int_0^t\int_{|z|<\epsilon}z\tilde M(dz, ds),
\end{align*}
converges in distribution to  a standard Brownian motion $W$, as $\epsilon$ tends to $0$. 
We refer to  \cite{MR1834755,MR2307403} for more details on this topic.
We remark that the robustness problem in this paper, does not rely on the above mentioned additional assumptions.

\subsubsection{The model}

Fix $p\in[2,\infty) $. We want to consider the following dynamical systems with memory and jumps in the setting of Section \ref{sec:functionalsOnS}:
\begin{align}\label{SFDEII}
X(t) &= \eta(0) + \int_0^t f(s,X_{s})ds + \int_0^t g(s, X_{s})dB(s) + \int_0^t \int_{\R_0} h_0(s,X_{s})\lambda(z) \tilde{N}(ds,dz), \ \ t\in [0,T]\nonumber\\
X_0 &= \eta\in S^p(\Omega,\mathcal{D})\, .
\end{align}
Here, 
\begin{align*}\label{def:h0}
h_0:\mathbf S^{\mathbb F}_p\rightarrow L^p(\Omega,  \mathbb R^{d\times k}),
\end{align*} and 
\begin{align*}
\lambda\in L^2(\nu, \mathbb R^{k\times n})\cap L^p(\nu, \mathbb R^{k\times n})
\end{align*} for some $k\in\mathbb N$.  
Observe that
\begin{align*}
h:=h_0 \lambda:\mathbf S^{\mathbb F}_p\rightarrow L^p(\Omega,  L^2(\nu,\mathbb R^{d\times n})).
\end{align*}
\begin{example}
Suppose that $n=k=d$ and that $h_0(t,\eta)$ and $\lambda$ are diagonal matrices. In particular that,
\begin{align*}
h_0(t,\eta)\lambda(z)
\end{align*}
is a diagonal matrix with entries $h^{i,i}_0(t,\eta)\lambda_{i,i}(z)$ for  $i=1, \ldots, n$. Then the component-wise form of the jump integral in the SFDDE \eqref{SFDEII} is given by
\begin{align*}
 \int_0^t\int_{\mathbb R_0}h^{i,i}_0(s,X_s)\lambda_{i,i}(z)\tilde{N}_i(ds,dz).
\end{align*}

If we let the delay parameter $r$ be equal to $0$, then  this example reduces to the problem of  \emph{robustness to model choice} treated in \cite{MR2916862}.
\end{example}

\noindent
Now, let us impose the following assumptions on  $f,g, h_0$ and $\lambda$:
\begin{assumption}\label{HYP:Rob}
\begin{description}
\item[(i)] The coefficient functionals $f$ and $g$ \color{black} are assumed to satisfy the assumptions  
$(\mathcal{P})$, $(\mathbf{D_1})$, $(\mathbf{D_2})$
\color{black} of Section \ref{sec:functionalsOnS}.
\item[(ii)] Whenever $Y$ is a \cadlag \ adapted process on $[-r,T]$, then $h_0(t, Y_t)$ is predictable.Moreover, the functional $h_0$  satisfies the Lipschitz and linear growth conditions:
\begin{align*}
\| h_0(t,\eta_1)-h_0(t,\eta_2)\|^p_{L^p(\Omega,  \mathbb R^{d\times k})}&\leq  L \| \eta_1-\eta_2 \|^p_{S^p(\Omega;\mathcal{D})} , \\
\| h_0(t,\eta)\|^p_{L^p(\Omega,  \mathbb R^{d\times k})}&\leq  K(1+ \| \eta \|^p_{S^p(\Omega;\mathcal{D})}).
\end{align*}
\end{description}
\end{assumption}
\noindent We claim now that the map $h:=h_0 \lambda$ satisfies the assumptions $(\mathcal{P})$, $(\mathbf{D_1})$, $(\mathbf{D_2})$ from  Section \ref{sec:functionalsOnS}.
In fact observe that
\begin{align*}
\begin{split}
&\| h(t,\eta_1)- h(t,\eta_2)\|^p_{L^p(\Omega, L^p(\nu))}+ \| h(t,\eta_1)- h(t,\eta_2)\|^p_{L^p(\Omega, L^2(\nu))} \\
&=\sum_{j=1}^{n}\| (h_0(t,\eta_1)- h_0(t,\eta_2))\lambda^{,j}\|^p_{L^p(\Omega, L^p(\nu_j))}+ \| (h_0(t,\eta_1)- h_0(t,\eta_2))\lambda^{,j}\|^p_{L^p(\Omega, L^2(\nu_j))}\\
&\leq  \| h_0(t,\eta_1)- h_0(t,\eta_2)\|^p_{L^p(\Omega;\mathbb{R}^{d\times k})}\sum_{j=1}^{n}\big(\| \lambda^{,j}\|^p_{L^p(\nu_j)}+\| \lambda^{,j}\|^p_{L^2(\nu_j)}\big)\\
&\leq L\| \eta_1-\eta_2 \|^2_{S^p(\Omega;\mathcal{D})}\big(\| \lambda\|^p_{L^p(\nu)}+ \| \lambda\|^p_{L^2(\nu)}\big).
\end{split}
\end{align*}
Thus $h$ satisfies the Lipschitz assumption $(\mathbf D_1)$. A similar argument yields $h$ satisfies the linear growth assumption $(\mathbf D_2)$. Thus, by the existence and uniqueness Theorem \ref{thm:existenceAndUniquenessEp}, the following result holds.
\begin{corollary}
The equation \eqref{SFDEII} has a unique  solution $^\eta X$ in $S^p_{ad}(\Omega; \mathcal{D}_T)$. Moreover, there exists  $D=D(K,\lambda,p,T,d,m,n)>0$, such that 
\begin{align}                                                                                                         
E\left[\sup_{-r\leq s\leq t} | {}^\eta X(s)|^p \right] \leq e^{Dt}(Dt+\|\eta\|_{S^p(\Omega;\mathcal{D})}^p),                                                                                                                          
\end{align}
for any $t\leq T$.
\end{corollary}

\subsubsection{The approximating model}

Let us first introduce some notation. For any $\varepsilon\in(0,1)$, define  $\lambda_\varepsilon(z)\in\mathbb R^{k\times n}$ by
\begin{align*}
\lambda_{\varepsilon}(z)=1_{\{|z|<\varepsilon\}}(z)\lambda(z),
\end{align*}for a.e. $z$.
Now, let $B$ be an $n$-dimensional $\mathbb F$-adapted Brownian motion, independent of $\tilde N$. Independence of $B$ and $W$ is not required, (see e.g. \cite{MR2814479}). We want to approximate equation \eqref{SFDEII} by replacing the integral with respect to the small jumps with an integral with respect to the Brownian motion $B$. More specifically, we will replace the integrators 
\begin{align}\label{eq:replacable integral}
\int_{\mathbb R_0}\lambda_{\varepsilon}^{i,j}(z) \tilde N^j(dt,dz)\, ,
\end{align} 
with the integrators
\begin{align}\label{eq:replacer inetgral}
\Lambda^{i,j}(\varepsilon)dB^j(t)\, ,
\end{align}
for $i = 1, \ldots, k; j=1, \ldots, n$.
Here, $\Lambda(\epsilon)
$ can be  any bounded deterministic function with values in $\mathbb R^{k\times n}$ converging to $0$ as $\varepsilon\rightarrow 0$. We choose to let $$\Lambda^{i,j}(\varepsilon)
=\| \lambda^{i,j}_{\varepsilon}\|_{L^2(\nu_j)}.$$
This choice corresponds to what has previously been used in the literature, see e.g. \cite{MR2814479}. A justification of this choice is considered in  Remark \ref{REM:P'} below. Notice now that
\begin{align*}
|\Lambda(\varepsilon)|_2=\| \lambda_{\varepsilon}\|_{L^2(\nu_j)}.
\end{align*}

\begin{remark}\label{REM:P'}
The choice $\Lambda^{i,j}(\epsilon)=\| \lambda^{i,j}_{\varepsilon}\|_{L^2(\nu_j)}$ above is reasonable in the sense that for a given predictable square integrable  process $Y$, this change of integrator preserves the variance of the integrals, i.e.
\begin{align*}
E\Big[\Big(\int_0^t \int_{\mathbb R_0}Y(s)\lambda_{\varepsilon}^{i,j}(z) \tilde N^j(ds,dz)\Big)^2 \Big]= E\Big[\int_0^t Y(s)^2ds \Big]
\| \lambda^{i,j}_{\varepsilon}\|^2_{L^2(\nu_j)}
=E\Big[\Big(\int_0^t Y(s)\Lambda^{i,j}(\epsilon) dB^j(s)\Big)^2 \Big],
 \end{align*} for $i= 1, \ldots, k; j=1, \ldots, n$.
From a financial terminology perspective where these models can be applied (see e.g. \cite{SMohammed, MR2814479, Chang2, Chang, MR2241374, MR2332878, kuchler}), this choice of $\Lambda$, preserves the \emph{total volatility}of a process, when \eqref{eq:replacable integral} is replaced by \eqref{eq:replacer inetgral}.
However, this particular choice of $\Lambda$ is not necessary for the analysis, as we will see in Remark \ref{Rem:generalLambda2} below. 
\end{remark}
Now, we are ready to exploit the dynamics of the \emph{approximated processes} $X^{\epsilon}$. 
Consider 
\begin{align}\label{aproxSFDE_1}
\begin{split}
X^{(\epsilon)}(t) &=\eta(0) + \int_0^t f(s,X_s^{\epsilon})ds + \int_0^t g(s, X_{s}^{\varepsilon})dW(s)\\
&\quad+ \int_0^t h_0(s,X_{s}^{\epsilon})\Lambda(\epsilon) dB(s)+ \int_0^t \int_{\mathbb{R}_0} h_0(s,X_{s}^{\epsilon})(\lambda(z)-\lambda_{\varepsilon}(s)) \tilde{N}(ds,dz)\\
X_0^{(\epsilon)}&=\eta
\end{split}
\end{align}
Before proceeding to the main result of this section, we make the following observations regarding the functionals in the approximated equation \eqref{aproxSFDE_1}:
\begin{itemize}
\item
The functionals
\begin{align*}
\begin{split}
\eta &\stackrel{g_1}{\longmapsto} h_0(t,\eta)\Lambda(\epsilon)\, ,\\
\eta &\stackrel{h_1}{\longmapsto} h_0(t,\eta)(\lambda-\lambda_{\varepsilon})\, ,
\end{split}
\end{align*}
satisfy the corresponding hypotheses from Section \ref{sec:functionalsOnS}
\item The Lipschitz and linear growth constant appearing in 
assumptions $(\mathbf{D_1})$ and $(\mathbf{D_2})$
can be chosen independent of $\epsilon$. In particular, 
we can deduce the following linear growth estimate:
\begin{align*}
\|& h_0(t,\eta)\Lambda(\epsilon)\|^p_{L^p(\Omega;\mathbb{R}^{d\times n})}+ \| h_0(t,\eta)(\lambda-\lambda_{\varepsilon})\|^p_{L^p(\Omega, L^p(\nu))}+\| h_0(t,\eta)(\lambda-\lambda_{\varepsilon})\|^p_{L^p(\Omega,L^2(\nu))}\\
&\leq \| h_0(t,\eta)\|^p_{L(\Omega; \mathbb{R}^{d\times k})}\sup_{\varepsilon\in(0,1)}|\Lambda(\epsilon)|^p+\| h_0(t,\eta)\|^p_{L^p(\Omega;\mathbb{R}^{d\times k})}\| \lambda\|^p_{L^p(\nu)}+\| h_0(t,\eta)\|^p_{L^p(\Omega;\mathbb{R}^{d\times k})}\| \lambda\|^p_{L^2(\nu)}\\
&\leq K'(1+\| \eta \|^p_{S^p(\Omega;\mathcal{D})}).
\end{align*}
A similar estimate holds for the Lipschitz condition $(\mathbf{D_1})$.
\end{itemize}
The following existence and uniqueness result immediately follows from Theorem \ref{thm:existenceAndUniquenessEp}.
\begin{corollary}\label{cor:existenceOfXepsilon}
For each $\epsilon>0$, there exists a unique strong solution $^\eta X^{\epsilon}$ to the equation \eqref{aproxSFDE_1}. Moreover, there exists a $D=D(K, \lambda, p,T,d,m,n)>0$, independent of $\epsilon$, such that
\begin{align}\label{aproxsolestimate}                                                                                                  E\left[\sup_{-r\leq s\leq t}|^\eta X^\epsilon(s)|^p\right]\leq e^{Dt}(Dt+\|\eta\|^p_{S^p(\Omega;\mathcal{D})})
\end{align}
for any $t\leq T$.
\end{corollary}
\noindent Now, we are ready to state the main result of the present section. This result guarantees that, when $\varepsilon$ tends to $0$, $X_\varepsilon$ converges to $X$ in $S^p_{ad}(\Omega; \mathcal{D}_T)$.

\begin{theorem}[Robustness]\label{thm:Robustness}
Suppose that $X$ satisfies equation \eqref{SFDEII} and $X^\epsilon$ satisfies equation \eqref{aproxSFDE_1}. Then there exist a constant $A:=A(p,T,\eta, K,L,\lambda)>0$, independent of $\epsilon$, such that
\begin{align}\label{ineq:Robustness}
E\left[\sup_{-r\leq s\leq t}|^\eta X(s)- {^\eta X}^{\epsilon}(s)|^p\right]\leq Ae^{At}(\|\lambda_\varepsilon\|^p_{L^2(\nu)}+ \|\lambda_\varepsilon\|^p_{L^p(\nu)}
).
\end{align}
\end{theorem}
 
\begin{proof}
Writing out the integral representation  of $X(s)$ and $X^{\epsilon}(s)$, we have that
\small\begin{align*}
X(s)-X^{\epsilon}(s)&=\int_0^s f(u,X_u)-f(u,X_u^{\epsilon}) du+ \int_0^s g(u,X_u)-g(u,X_u^{\epsilon})dW(u)\\
&+\int_0^s \int_{\mathbb R_0^d}(h_0(u,X_u)-h_0(u,X_u^{\epsilon}))\lambda(z)+ h_0(u,X^{\epsilon}_{u})\lambda_{\varepsilon}(z) \tilde{N}(du,dz)\\
&- \int_0^s h_0(u,X_{u}^{\epsilon})\Lambda_p(\epsilon)  dB(u),\\
X_0-X_0^{{\epsilon}}&= 0.
\end{align*}\normalsize
Let us first consider some estimates for the integrands of $\tilde N$ and $B$. Observe that
\begin{align*}
\| (h_0(u,X_u)&-h_0(u,X_u^{\epsilon}))\lambda+h_0(u,X_u^{\epsilon})\lambda_{\varepsilon}\|_{L^p(\Omega;L^p(\nu))}\\
&\leq
\| (h_0(u,X_u)-h_0(u,X_u^{\varepsilon})\|_{L^p(\Omega;\mathbb{R}^{d\times k})}\|\lambda\|_{L^p(\nu)}+\| h_0(u,X_u^{\epsilon})\|_{L^p(\Omega;\mathbb{R}^{d\times k})}\| \lambda_{\varepsilon}\|_{L^p(\nu)}
\\
&\leq L^{1/p}\| X_u-X_u^{\epsilon}\|_{S^p(\Omega;\mathcal{D})}\|\lambda\|_{L^p(\nu)} +K^{1/p}(1+\| X_u^{\epsilon}\|_{S^p(\Omega;\mathcal{D})})^{1/p}\| \lambda_{\varepsilon}\|_{L^p(\nu)},
\end{align*}
and hence
\begin{align}\label{HestimatE_1}
\begin{split}
\| (h_0(u,X_u)&-h_0(u,X_u^{\epsilon}))\lambda+h_0(u,X_u^{\epsilon})\lambda_{\varepsilon}\|^p_{L^p(\Omega, L^p(\nu))}\\
&\leq 2^{p-1}\Big(L\| X_u-X_u^{\epsilon}\|^p_{S^p(\Omega;\mathcal{D})}\|\lambda\|^p_{L^p(\nu)} +K(1+\| X_u^{\epsilon}\|^p_{S^p(\Omega;\mathcal{D})})\|_{L^p(\Omega)}\| \lambda_{\varepsilon}\|^p_{L^p(\nu)}\Big).\\
\end{split}
\end{align}
In an analogous manner we have that
\begin{align}\label{HestimatE_2}
\begin{split}
\| (h_0(u,X_u)&-h_0(u,X_u^{\epsilon}))\lambda+h_0(u,X_u^{\epsilon})\lambda_{\varepsilon}\|^p_{L^p(\Omega,L^2(\nu))}\\
&\leq 2^{p-1}\Big(L\| X_u-X_u^{\epsilon}\|^p_{S^p(\Omega;\mathcal{D})}\|\lambda\|^p_{L^2(\nu)} +K(1+\| X_u^{\epsilon}\|^p_{S^p(\Omega;\mathcal{D})})\|\lambda_{\varepsilon}\|^p_{L^2(\nu)}
\Big),
\end{split}
\end{align}
and that
\begin{align*}
 \| h_0(u,X_u^{\epsilon})\Lambda(\varepsilon)\|^p_{L^p(\Omega;\mathbb{R}^{d\times n})}\leq K(1+\| X_u^{\epsilon}\|^p_{S^p(\Omega;\mathcal{D})})
|\Lambda(\varepsilon)|^p
\end{align*}
Using Lemma \ref{lem:Kunita}, the Lipschitz condition $(\mathbf{D_1})$, estimates \eqref{HestimatE_1}, \eqref{HestimatE_2}, and Corollary \ref{cor:existenceOfXepsilon} we have that there exist a constant $D':=D'(p,K,L,\lambda)$, independent of $\varepsilon$ such that
\begin{align*}\begin{split}
\alpha_{\varepsilon}(t):&=\mathbb{E}\big[\sup_{-r\leq s\leq t}|^\eta X(s)-^\eta X^{\epsilon}(s)|^p\big]\\
&\leq\int_0^t \Big(\| f(u,X_u)-f(u,X_u^{\epsilon})\|^p_{L^p(\Omega;\mathbb{R}^d)}
+ \| g(u,X_u)-g(u,X_u^{\epsilon})\|^p_{L^p(\Omega;\mathbb{R}^{d\times n})}\\
&+\| (h_0(u,X_u)-h_0(u,X_u^{\epsilon}))\lambda+h_0(u,X_u^{\epsilon})\lambda_{\varepsilon}\|^p_{L^p(\Omega;L^p(\nu))}\\
&+\| (h_0(u,X_u)-h_0(u,X_u^{\epsilon}))\lambda+h_0(u,X_u^{\epsilon})\lambda_{\varepsilon}\|^p_{L^p(\Omega,L^2(\nu))}\\
&+\| h_0(u,X^{\epsilon}_u)\Lambda(\varepsilon)\|^p_{L^p(\Omega;\mathbb{R}^{d\times n})}\Big)du
\end{split}\\
&\leq D'\int_0^t   \Big(\| X_u-X_u^{\epsilon}\|^p_{S^p(\Omega;\mathcal{D})}+\big(1+\| X_u^{\epsilon}\|^p_{S^p(\Omega;\mathcal{D})}\big)\big(\|\lambda_{\varepsilon}\|^p_{L^2(\nu)}+\|\lambda_{\varepsilon}\|^p_{L^p(\nu)}+|\Lambda(\varepsilon)|^p \big) \Big)du\\
&\leq D'\int_0^t \alpha_{\varepsilon}(u)du +tD'\big(1+e^{Dt}(Dt+\|\eta\|^p_{S^p(\Omega;\mathcal{D})})\big)\big(\|\lambda_{\varepsilon}\|^p_{L^2(\nu)}+\|\lambda_{\varepsilon}\|^p_{L^p(\nu)}+|\Lambda(\varepsilon)|^p\big).
\end{align*}
Now, set $B_t := tD'\big(1+e^{Dt}(Dt+\|\eta\|^p_{S^p(\Omega;\mathcal{D})})\big)$ which is a non-decreasing function in $t$ and hence by Gr\"{o}nwall's inequality, it follows that 
\begin{align*}
 \alpha_{\varepsilon}(t)\leq B_t e^{D't}\big(\|\lambda_{\varepsilon}\|^p_{L^2(\nu)}+\|\lambda_{\varepsilon}\|^p_{L^p(\nu)}+|\Lambda(\varepsilon)|^p\big).
\end{align*}
Since $|\Lambda(\varepsilon)|^p=\|\lambda_{\varepsilon}\|^p_{L^2(\nu)}$, the result holds with $A:=\max\{2B_T, D'\}$.
\end{proof}
\begin{remark}\label{Rem:generalLambda2}
We have chosen to scale the Brownian motions $B^j$ in  the equation \eqref{aproxSFDE_1}  for $X^{\epsilon}$ with $\Lambda^{i,j}(\epsilon):=\|\lambda^{i,j}_\varepsilon\|_{L^2(\nu_j)}$. However, if we return to \eqref{eq:replacable integral}-\eqref{eq:replacer inetgral}, we could let $\Lambda_\epsilon$ be any $\mathbb R^{k\times n}$-valued function $\Lambda(\varepsilon)\geq 0$, $\varepsilon\geq 0$, bounded from above and converging to $0$ as $\varepsilon \rightarrow0$. Corollary \ref{cor:existenceOfXepsilon} and Theorem \ref{thm:Robustness} still hold,  with the inequality \eqref{ineq:Robustness} replaced by
\begin{align}
 E\left[\sup_{-r\leq s\leq t}|^\eta X(s)-   {^\eta  X}^{\epsilon}(s)|^p\right]\leq A'e^{A't}(\|\lambda_{\varepsilon}\|^p_{L^2(\nu)}+\|\lambda_{\varepsilon}\|^p_{L^p(\nu)}+|\Lambda(\varepsilon)|^p).
\end{align}
This can be easily seen by reexamining the proofs of Corollary \ref{cor:existenceOfXepsilon} and Theorem \ref{thm:Robustness}.
\end{remark}

\section{It\^{o}'s formula}\label{SEC:Ito}

In this section we aim at deriving It\^{o}'s formula for the SFDDE's studied in Section \ref{TheEquationsExistenceUniquenessMomentestimates3}, which we recall, have the form \eqref{eq:Main1},
\begin{align*}
\begin{split}
 d X(t)&=f(t, X_t) d t+g(t,X_t)d W(t)+ \int_{\mathbb{R}_0}h(t,X_t,z)\tilde N(d t, d z)\\
X_0&=\eta \, ,
\end{split}
\end{align*}
where $X_t$ is the segment of the process $X$ in $[t-r,t]$, with $r>0$ a finite delay, taking values in a suitable path space, and $X(t) \in \R^d$ the present value of the process $X$. For the whole section we work in the $M^p$-framework. See Section \ref{sec:M2existenceAndUniqueness}. 
Moreover, we assume that $f$, $g$ and $h$ are \emph{deterministic} functionals, see Remark \ref{REM:DetF}.

The main problem, when dealing with the SFDDE \eqref{eq:Main1} is that the infinite dimensional process $(X_t)_{t \in [0,T]}$ fails, in general, to be a semimartingale and standard It\^{o} calculus does not apply. In order to overcome this problem several approaches have been used in the literature. 

The first attempts go back to \cite{MR1652338,Yan} where an It\^{o}-type formula for continuous SFDDE was proved via Malliavin calculus. More recently, exploiting the concepts of \textit{horizontal derivative} and \textit{vertical derivative}, a path-wise It\^{o} calculus for non-anticipative stochastic differential equation was derived in \cite{cont2010change,cont2013functional}. 

In \cite{MR3272616}, an It\^{o} formula for Banach-valued continuous processes was proved exploiting the \textit{calculus via regularisation}, where an application to window processes (see \cite[Definition 1.4]{MR3272616}) is also provided. Several works have followed studying It\^{o}-type formulae for delay equations exploiting the \textit{calculus via regularisation} and showing that the Banach-valued setting and the path-dependent setting can be in fact connected, see, e.g. \cite{c14cal,c15fun}. Eventually, the connection between the Banach space stochastic calculus and the path-wise calculus was made definitely clear in \cite{c14reg,f13in}.

We remark that the literature on It\^{o} formulae by the calculus via regularisation deals with equations driven by continuous noises. In this paper, we focus on the SFDDE's with jumps, thus extending the existing literature in this respect. We have chosen to consider the approach of the \textit{calculus via regularisation}, first introduced in \cite{MR1350257,MR1367665}, which was proved to be well-suited when dealing with infinite-dimensional processes or in the non-semimartingale case, see e.g. \cite{c14cal,c14reg,coviello2011stochastic,di2012generalized,MR3272616}. In particular, we prove an It\^{o} formula for the SFDDE \eqref{eq:Main1} with values in $M^p$ and we show that our result is coherent with those of \cite{MR1652338,Yan}. In the Appendix we provide a connection with the path-dependent calculus developed in \cite{cont2010change,cont2013functional}.

Recall that, for a finite delay $r>0$, $L^p:=L^p([-r,0];\R^d)$ endowed with the standard norm $\|\cdot\|_{L^p}$, $p\in [2,\infty)$. In what follows we implicitly identify the topological dual of $L^p$, i.e. $ (L^p)^\ast$, with $L^q$ being $\frac{1}{p}+\frac{1}{q}=1$, via the natural isomorphism given by
\begin{align*}
J:&L^q \rightarrow (L^p)^\ast\\
g &\mapsto J(g)= \Scal{g,\cdot},
\end{align*}
where $J(g)$ acts on $L^p$ as follows
$$J(g)(f) = \Scal{g,f} = \int_{-r}^0 g(s)\cdot f(s)ds, \quad g\in L^q, \quad f\in L^p,$$
being $\cdot$ in the integral the usual scalar product in $\R^d$. It is well-known that $J$ is a continuous linear isomorphism and hence, with a slight abuse of notation, we just write $h\in (L^p)^\ast$ when we actually mean $J^{-1}(h)\in L^q$, i.e. $(L^p)^\ast \cong L^q$.

Moreover, we denote by $C^1(L^p)$ the space of once Fr\'{e}chet differentiable, not necessarily linear, functionals $F:L^p\rightarrow \R$ with continuous derivative, that is $DF:L^p \rightarrow L(L^p,\R)$ where $L(L^p,\R)$ denotes the space of continuous linear operators from $L^p$ to $\R$. Now, since $F$ is $\R$-valued, we actually have that $L(L^p,\R)=(L^p)^\ast$. Hence we can regard $DF(f)$, $f\in L^p$ as an element in $L^q$ via $J^{-1}$. In a summary, we identify $DF(f)$ with $J^{-1}(DF(f))$ and simply write $DF: L^p \rightarrow L^q$ so that
$$DF(f)(g) = {}_q \langle DF(f),g\rangle_p = \int_{-r}^0 DF(f)(s)\cdot g(s) ds, \quad f\in L^p, \quad g\in L^q,$$
where the first equality is, by an abuse of notation, meant as an identification.

Also, recall that $|\cdot|$ denotes the Euclidean norm in $\R^d$. Finally, recall that $M^p= L^p \times \R^d$ endowed with the standard product norm.

Let the SFDDE
\begin{equation}\label{EQN:SFDDE1}
\begin{cases}
d X(t) = f(t,X_t,X(t)) dt + g(t,X_t,X(t)) dW(t) + \int_{\R_0}h(t,X_t,X(t))(z) \tilde{N}(dt,dz)\, ,\\
(X_0,X(0))= (\eta,x) \in M^p \,,
\end{cases}
\end{equation}
for $t\in [0,T]$, $T<\infty$. We assume that $f:[0,T]\times M^p \to \R^d$, $g:[0,T]\times M^p \to \R^{d \times m}$ and $h:[0,T]\times M^p \times \R_0 \to \R^{d \times n}$ satisfy Assumptions $(\mathbf{L_1})$ and $(\mathbf{L_2})$ so that Theorem \ref{thm:existenceAndUniqueness_II} holds and equation \eqref{EQN:SFDDE1} admits a unique strong solution.

In the sequel every process is indexed by the time $t \in [0,T]$ and following \cite{MR3272616}, if necessary, we extend the process $X = (X(t))_{t \in [0,T]}$ to the positive real line as follows
\[
X(t) := 
\begin{cases}
X(t) & \mbox{ for } t \in [0,T]\, ,\\
X(T) & \mbox{ for } t > T\, ,\\
\end{cases}
\, .
\]
Next, we consider the definition of forward integral.

\begin{definition}\label{DEF:Xseps}
Let $X = \{X(s), s \in [0,T]\}$ and $Y= \{Y(s), s \in [0,T]\}$ two $\RR^d-$valued process. For every $ t \in [0,T]$ we define the forward integral of $Y$ w.r.t. $X$ by $\int_0^t Y(s)\, \cdot  dX(s)$ as the following limit,
\begin{equation}\label{EQN:Int}
\int_0^t Y(s)\, \cdot  dX(s):= \lim_{\epsilon \downarrow 0} \int_0^t Y(s)\,\cdot  \frac{X(s+\epsilon)-X(s)}{\epsilon}ds\, ,
\end{equation}
where the convergence is uniformly on compacts in probability (ucp).

Similarly, let $X = \{X_s, s \in [0,T]\}$, and $Y= \{Y_s, s \in [0,T]\}$, be $L^p$-valued and $L^q$-valued processes, respectively. For every $ t \in [0,T]$ we define the $L^p-$forward integral of $Y$ w.r.t. $X$ as the following limit,
\begin{equation}\label{EQN:Int2}
 \int_0^t \int_{-r}^0 Y_s(\theta) \cdot dX_s(\theta) := \lim_{\epsilon \searrow 0} \int_0^t \int_{-r}^0 Y_s(\theta)\, \cdot  \frac{X_{s+\epsilon}(\theta)-X_s(\theta)}{\epsilon} d \theta ds \, ,
\end{equation}
where the convergence is uniformly on compacts in probability. We introduce the short-hand notation:
\[
\int_0^t {}_q\langle Y_s, dX_s \rangle_p :=  \int_0^t \int_{-r}^0 Y_s(\theta) \cdot dX_s(\theta) \, .
\]
\end{definition}

Recall that a sequence of real-valued processes $\{X^n\}_{n \geq 1}$ converges to a process $X$ \textit{uniformly on compacts in probability} (ucp), if for each $t \in [0,T]$ we have that
\[
\sup_{0 \leq s \leq t}|X^n_s - X_s| \to 0\, ,
\]
in probability. See e.g. \cite[p.57]{MR2273672}. In this section, if not otherwise stated, any limit will be taken in the ucp sense.

\begin{remark}\label{REM:ucpC}
Following \cite{MR3272616}, in Definition \ref{DEF:Xseps} we have considered the ucp limit. In fact the space of c\`{a}dl\`{a}g functions is a metrizable space with the metric induced by the \textit{ucp} topology, see, e.g. \cite[p.57]{MR2273672}. This implies that, being the approximating sequence in the right-hand-side of equation \eqref{EQN:Int} \cadlag, the ucp convergence ensures that the limiting process, that is the forward integral, is also \cadlag.
\end{remark}

Let us now introduce the notation we will use in the present work.

\begin{definition}\label{DEF:Der}
Let $F: [0,T] \times L^p \times \R^d \to \R$ a given function, we say that 
\[
F \in C^{1,1,2}\left ([0,T] \times L^p \times \R^d \right )\, ,
\]
if $F$ is continuously differentiable w.r.t. the first variable, Fr\'{e}chet differentiable with continuous derivative w.r.t. the second variable, and twice continuously differentiable w.r.t. the third variable. 

We thus denote by $\partial_t$ the derivative w.r.t. to time, $D_iF$ the Fr\'{e}chet derivative w.r.t. the $i$-th component of the segment $X_t$ and $\partial_i$ the derivative w.r.t. the $i$-th component of the present state $X(t)$ and finally, $\partial_{i,j}$ the second order derivative w.r.t. the $i,j$-th component of $X(t)$.

We will then define the Fr\'{e}chet gradient w.r.t. the segment as
\[
D := \left (D_1, \dots,D_d\right )\, ,
\] 
the gradient w.r.t. the present state
\[
\nabla_x := \left (\partial_{1},\dots, \partial_{d}\right )\, ,
\]
and the Hessian matrix w.r.t. the present state
\[
\nabla_x^2 := \left (\partial_{i,j}\right )_{i,j=1,\dots,d}\, .
\]
\end{definition}

\begin{definition}\label{DEF:WeRi}
Let $\eta \in W^{1,p}([-r,0];\R^d)=:W^{1,p}$, then we define by $\partial_{\theta,i} \eta$ and $\partial_{\theta,i}^+ \eta$ the weak derivative and the right weak derivative, respectively, of the $i$-th component of $\eta$. Accordingly we define the gradient as
\[
\nabla_\theta := \left (\partial_{\theta,1},\dots, \partial_{\theta,d}\right )\, ,\quad \mbox{ resp. } \nabla^+_\theta := \left (\partial^+_{\theta,1},\dots, \partial^+_{\theta,d}\right )\, .
\]
\end{definition}

Eventually, in proving It\^{o}'s formula, we will need the notion of \textit{modulus of continuity} of operators between infinite-dimensional normed spaces.

\begin{definition}[Modulus of continuity]\label{DEF:MC}
Let $\left (Y_1, \| \cdot \|_{Y_1}\right )$ and $\left (Y_2, \| \cdot \|_{Y_2}\right )$ be two normed spaces and $F: Y_1 \to Y_2$ a uniformly continuous function. We define the \textit{modulus of continuity} of $F$ as
\[
\varpi(\epsilon) := \sup_{\|y-y'\|_{Y_1} \leq \epsilon} \| F(y) -F(y')\|_{Y_2}, \quad \epsilon >0 .
\]
\end{definition}

We thus have the following It\^{o}'s formula for SFDDE \eqref{EQN:SFDDE1}.

\begin{theorem}[It\^{o}'s formula]\label{THM:Ito}
Let $X$ be the solution to equation \eqref{EQN:SFDDE1} and let $F:[0,T] \times L^p \times \R^d \to \R$ such that $F \in C^{1,1,2}\left ([0,T] \times L^p \times \R^d \right )$ and such that $DF(t,\eta,x)\in W^{1,q}$, ($q$ such that $\frac{1}{p}+\frac{1}{q}=1$) for any $t\in [0,T]$, $\eta\in L^p$ and $x\in \R^d$ and $\nabla_{\theta}DF(t,\cdot,x):L^p \rightarrow L^q$ is uniformly continuous. Then the following limit exists in the ucp sense,
\begin{align}\label{EQN:ForInt}
\lim_{\epsilon \searrow 0} \frac{1}{\epsilon} \int_0^t {}_q \langle DF(s,X_s,X(s)), X_{s+\epsilon}-X_s\rangle_p \,ds =: \int_0^t {}_q\langle DF(s,X_s,X(s)), d X_s\rangle_p.
\end{align}
Moreover, for $t\in [0,T]$, we have that
\begin{align}\label{EQN:Ito}
\begin{split}
&F(t,X_t,X(t))=F(0,X_0,X(0)) +\int_0^t \partial_t F(s,X_s,X(s)) ds+ \int_0^t \Scal{DF(s,X_s,X(s))\, , \, d X_s} \\
&+ \int_0^t \nabla_x F(s,X_s,X(s))\cdot d X (s) +\frac{1}{2} \int_0^t Tr \left [g^*(s,X_s,X(s)) \nabla_x^2 F(s,X_s,X(s))g(s,X_s,X(s))\right ] d s \\
&+ \int_0^t \int_{\R_0} \left (F(s,X_s,X(s)+h(s,X_s,X(s))(z))-F(s,X_s,X(s))\right ) N(ds,dz) \\
&-  \int_0^t \int_{\R_0} \nabla_x F(s,X_s,X(s)) h(s,X_s,X(s))(z) N(ds,dz)\, ,
\end{split}
\end{align}
holds, $P$-a.s., where we have denoted by $Tr$ the trace and by $g^*$ the adjoint of $g$ and the fourth term in equation \eqref{EQN:Ito} has to be intended component-wise, that is
\[
\begin{split}
&\int_0^t \int_{\R_0} \left (F(s,X_s,X(s)+h(s,X_s,X(s))(z))-F(s,X_s,X(s))\right ) N(ds,dz) \\
&:= \sum_{i=1}^n \int_0^t \int_{\R_0} \left (F(s,X_s,X(s)+h^{\cdot, i}(s,X_s,X(s))(z))-F(s,X_s,X(s))\right ) N^i(ds,dz)\, .
\end{split}
\]
\end{theorem}
\begin{proof}
Let $t\in [0,T]$. First, observe that for $\varepsilon>0$ small enough, we have
\begin{align}\label{EQN:F}
\begin{split}
&\frac{1}{\epsilon} \int_0^t F(s+\epsilon,X_{s+\epsilon},X(s+\epsilon))-F(s,X_s,X(s))\, ds =\\
&\frac{1}{\epsilon} \int_\epsilon^{t+\epsilon} F(s,X_{s},X(s)) ds - \frac{1}{\varepsilon}\int_0^t F(s,X_s,X(s))\, ds =\\
&=\frac{1}{\epsilon} \int_t^{t+\epsilon} F(s,X_s,X(s)) ds - \frac{1}{\epsilon} \int_0^\epsilon F(s,X_s,X(s)) ds \, ,
\end{split}
\end{align}
which, by the continuity of $F$ and $X_s$ and the right-continuity of $X(s)$, $s\in [0,T]$, recalling remark \ref{REM:ucpC} and arguing as in \cite[eq. (4.6)]{BanRus}, converges ucp to 
\begin{align*}
F(t,X_t,X(t))-F(0,X_0,X(0))\, .
\end{align*}
The first part of \eqref{EQN:F} can be rewritten as
\begin{align*}
&\frac{1}{\epsilon} \int_0^t F(s+\epsilon,X_{s+\epsilon},X(s+\epsilon))-F(s,X_s,X(s))\, ds \\
&\hspace{2cm}= \underbrace{\frac{1}{\epsilon} \int_0^t F(s+\epsilon,X_{s+\epsilon},X(s+\epsilon))-F(s,X_{s+\epsilon},X(s+\epsilon))\, ds}_{J^1_\epsilon} \\
&\hspace{3cm}+\underbrace{\frac{1}{\epsilon} \int_0^t F(s,X_{s+\epsilon},X(s+\epsilon))-F(s,X_{s+\epsilon},X(s))\, ds}_{J^2_\epsilon} \\
&\hspace{3cm}+ \underbrace{\frac{1}{\epsilon} \int_0^t F(s,X_{s+\epsilon},X(s))-F(s,X_s,X(s))\, ds}_{J^3_\epsilon}\, .
\end{align*}
Following the same arguments as in the proof of \cite[Theorem 5.2]{MR3272616} we can show that
\[
\lim_{\epsilon \searrow 0} J^1_\epsilon = \int_0^t \partial_t F(s,X_s,X(s)) ds, \quad \mbox{ucp}.
\]
Let us now consider $J^2_\epsilon$. A straightforward application of \cite[Corollary 4.4]{BanRus} implies that 
\begin{align*}
\begin{split}
J^2_\epsilon \to & \int_0^t \nabla_x F(s,X_s,X(s)) \cdot  d X (s) \\
&+\frac{1}{2} \int_0^t Tr \left [g^*(s,X_s,X(s)) \nabla_x^2 F(s,X_s,X(s))g(s,X_s,X(s))\right ] d s \\
&+\int_0^t \int_{\R_0} \left (F(s,X_s,X(s)+h(s,X_s,X(s),z))-F(s,X_s,X(s))\right ) N(ds,dz) \\
&-  \int_0^t \int_{\R_0} \nabla_x F(s,X_s,X(s)) h(s,X_s,X(s),z) N(ds,dz)\, , \quad \mbox{ as } \epsilon \searrow 0 \, .
\end{split}
\end{align*}

Let us now show that
\begin{align*}
\lim_{\epsilon \searrow 0} J^3_\epsilon = \int_0^t \Scal{DF(s,X_s,X(s))\, , \, d X_s} ds\,.
\end{align*}
By an application of the infinite-dimensional version of Taylor's theorem of order one (see e.g. \cite[Ch. 4, Theorem 4.C]{Zeidler95}), we obtain
\begin{equation}\label{EQN:Tay}
\begin{split}
&\frac{1}{\epsilon} \int_0^t F(s,X_{s+\epsilon},X(s))-F(s,X_s,X(s))\, ds  \\
=& \frac{1}{\epsilon}\int_0^t \int_0^1 \Scal{DF(s,X_{s} + \tau (X_{s+\epsilon} - X_s),X(s)) \, ,  \,  X_{s+\epsilon} -X_{s}} \, d \tau ds \\
=&\frac{1}{\epsilon}\int_0^t \int_0^1 \int_{-r}^0 DF(s,X_{s} + \tau (X_{s+\epsilon} - X_s),X(s))(\alpha) \cdot (X(s+\epsilon+\alpha) -X(s+\alpha)) d \alpha d\tau ds \\
=&-\int_0^t \underbrace{\frac{1}{\epsilon} \int_0^1 \int_{-r}^0 \left (DF(s,X_{s,s+\epsilon}^\tau,X(s) )(\alpha+\epsilon)-DF(s,X_{s,s+\epsilon}^\tau,X(s) )(\alpha)\right ) \cdot  X(s+\epsilon+\alpha) d \alpha d \tau}_{J^{3,1}_\epsilon} ds\\
&+\int_0^t\underbrace{  \frac{1}{\epsilon} \int_0^1 \int_{-r}^0 DF(s,X_{s,s+\epsilon}^\tau,X(s))(\alpha+ \epsilon)\cdot  X(s+\alpha+\epsilon) d \alpha d \tau ds}_{J^{3,2}_\epsilon} \\
&- \int_0^t\underbrace{ \frac{1}{\epsilon} \int_0^1 \int_{-r}^0 DF(s,X_{s,s+\epsilon}^\tau,X(s))(\alpha)\cdot  X(s+\alpha) d \alpha d \tau}_{J^{3,2}_\epsilon}ds \, ,
\end{split}
\end{equation}
where we have denoted by $X_{s,s+\epsilon}^\tau := X_{s} + \tau (X_{s+\epsilon} - X_s)$. We apply the change of variables $g(\alpha) = \alpha + \epsilon$ to the first term of $J^{3,2}_\epsilon$ in Equation \eqref{EQN:Tay} in order to obtain
\begin{align*}
J^{3,2}_\epsilon =&\, \frac{1}{\epsilon} \int_0^1 \int_{-r+\epsilon}^\epsilon DF(s,X_{s,s+\epsilon}^\tau,X(s))(\alpha) \cdot X(s+\alpha) d \alpha d \tau \\
&-  \frac{1}{\epsilon} \int_0^1 \int_{-r}^0 DF(s,X_{s,s+\epsilon}^\tau, X(s))(\alpha)\cdot  X(s+\alpha) d \alpha d \tau \\
=&\, \underbrace{\frac{1}{\epsilon} \int_0^1 \int_{0}^\epsilon DF(s,X_{s,s+\epsilon}^\tau,X(s))(\alpha)\cdot  X(s+\alpha) d \alpha d \tau }_{J^{3,2,1}_\epsilon}\\
& -\underbrace{\frac{1}{\epsilon} \int_0^1 \int_{-r}^{-r+\epsilon} DF(s,X_{s,s+\epsilon}^\tau,X(s))(\alpha)\cdot  X(s+\alpha) d \alpha d \tau}_{J^{3,2,2}_\epsilon}\, .
\end{align*}
We thus have, from the continuity of $DF$ and $X_s$, that
\[
\lim_{\epsilon \searrow 0 } J^{3,2,1}_\epsilon = DF(s,X_s,X(s))(0)\cdot  X(s)\, ,\quad \lim_{\epsilon \searrow 0 } J^{3,2,2}_\epsilon = DF(s,X_s,X(s)) (-r)\cdot  X(s-r)\, .
\]

Let $\nabla_{\theta}DF(s,X_{s,s+\epsilon}^\tau,X(s))$ denote a version of the weak derivative of $DF(s,X_{s,s+\epsilon}^\tau,X(s))\in W^{1,q}$. Consider $J^{3,1}_\epsilon$. Using the mean value-theorem and interchanging the order of integration by Fubini's theorem we have
\begin{equation}\label{EQN:J2Etc}
\begin{split}
J^{3,1}_\epsilon =& \, \frac{1}{\epsilon} \int_0^1 \int_{-r}^0  \int_{\alpha}^{\alpha+\epsilon} \nabla_\theta DF(s,X_{s,s+\epsilon}^\tau,X(s))(\beta)\,  d\beta   \cdot  X(s+\epsilon +\alpha) d \alpha d \tau \\
=& \, \frac{1}{\epsilon} \int_0^1 \int_{-r}^{-r+\epsilon}  \int_{-r}^{\beta} \nabla_\theta DF(s,X_{s,s+\epsilon}^\tau,X(s))(\beta) \cdot  X(s+\epsilon+\alpha) \, d\alpha d\beta d \tau \\
&+\frac{1}{\epsilon} \int_0^1 \int_{-r+\epsilon}^0 \int_{\beta-\epsilon}^{\beta} \nabla_\theta DF(s,X_{s,s+\epsilon}^\tau,X(s))(\beta) \cdot  X(s+\epsilon+\alpha) \, d \alpha d\beta d \tau \\
&+ \frac{1}{\epsilon} \int_0^1 \int_{0}^{\epsilon}  \int_{\beta-\epsilon}^{0} \nabla_\theta DF(s,X_{s,s+\epsilon}^\tau,X(s))(\beta)\cdot  X(s+\epsilon+\alpha) \, d \alpha d\beta d \tau.
\end{split}
\end{equation}

Now, we add and subtract integral terms so that the second integral on the right-hand side of \eqref{EQN:J2Etc} goes from $-r$ to $0$, that is

\begin{equation}\label{EQN:J2Etcbis}
\begin{split}
J^{3,1}_\epsilon &= \frac{1}{\epsilon} \int_0^1 \int_{-r}^0  \int_{\beta-\epsilon}^{\beta} \nabla_\theta DF(s,X_{s,s+\epsilon}^\tau,X(s))(\beta)\cdot  X(s+\epsilon+\alpha) \, d\alpha d\beta d \tau \\
&+\frac{1}{\epsilon} \int_0^1 \int_{0}^\epsilon \int_{\beta-\epsilon}^{0} \nabla_\theta DF(s,X_{s,s+\epsilon}^\tau,X(s))(\beta) \cdot  X(s+\epsilon+\alpha) \, d \alpha d\beta d \tau \\
&- \frac{1}{\epsilon} \int_0^1 \int_{-r}^{-r +\epsilon}  \int_{\beta-\epsilon}^{-r} \nabla_\theta DF(s,X_{s,s+\epsilon}^\tau,X(s))(\beta) \cdot  X(s+\epsilon+\alpha) \, d \alpha d\beta d \tau.
\end{split}
\end{equation}
Then, Lebesgue's differentiation theorem implies that the second and third integral on the right-hand side of \eqref{EQN:J2Etcbis} converge to $0$ as $\epsilon \to 0$. Concerning the first integral on the right-hand side of \eqref{EQN:J2Etcbis}, we add and subtract the corresponding terms in order to have
\begin{align}\label{EQN:J2Etc2}
\begin{split}
&\frac{1}{\epsilon} \int_0^1 \int_{-r}^0  \int_{\beta-\epsilon}^{\beta} \nabla_\theta DF(s,X_{s,s+\epsilon}^\tau,X(s))(\beta) \cdot  X(s+\epsilon+\alpha) \, d \alpha d\beta d \tau \\
&= \underbrace{ \frac{1}{\epsilon} \int_0^1 \int_{-r}^0 \left (\nabla_\theta DF(s,X_{s,s+\epsilon}^\tau,X(s))(\beta)-\nabla_\theta DF(s,X_{s},X(s))(\beta)\right )\cdot \left ( \int_{\beta-\epsilon}^{\beta} X(s+\epsilon+\alpha) d \alpha \right )  d\beta d \tau}_{J^{3,1,1}_\epsilon} \\
&+ \underbrace{ \frac{1}{\epsilon} \int_0^1 \int_{-r}^0  \nabla_\theta DF(s,X_{s},X(s))(\beta) \cdot \left ( \int_{\beta-\epsilon}^{\beta} X(s+\epsilon+\alpha) d \alpha \right ) d\beta d \tau }_{J^{3,1,2}_\epsilon}\, .
\end{split}
\end{align}
Using H\"{o}lder's inequality with exponents, $p,q\geq 2$, $\frac{1}{p}+\frac{1}{q}=1$ on $J^{3,1,1}_\epsilon$ we have
\begin{align*}
&|J^{3,1,1}_\epsilon|\\
&\leq \int_0^1 \|\nabla_\theta DF(s,X_{s,s+\epsilon}^\tau,X(s))-\nabla_\theta DF(s,X_{s},X(s))\|_{L^q} d\tau \left(\int_{-r}^0 \left|\frac{1}{\epsilon} \int_{\beta}^{\beta+\epsilon} X_s(\alpha) d \alpha \right|^p d\beta\right)^{1/p}\\
&\leq \int_0^1 \|\nabla_\theta DF(s,X_{s,s+\epsilon}^\tau,X(s))-\nabla_\theta DF(s,X_{s},X(s))\|_{L^q} d\tau \, \|M[X_s]\|_{L^p},
\end{align*}
where
$$M[X_s](\beta):= \sup_{\epsilon >0}\frac{1}{\epsilon} \int_{\beta}^{\beta+\epsilon} \left|X_s(\alpha)\right| d \alpha,$$
is the Hardy-Littlewood maximal operator, which is a (non-linear) bounded operator from $L^p$ to $L^p$, $p>1$. Hence, we can apply Lebesgue's dominated convergence theorem and the fact that by Lebesgue's differentiation we have
$$\lim_{\epsilon \searrow 0}\left(\int_{-r}^0 \frac{1}{\epsilon} \int_{\beta-\epsilon}^{\beta} \left|X(s+\epsilon+\alpha)\right|^p d \alpha   d\beta\right)^{1/p} = \|X_s\|_{L^p}.$$
The above arguments in connection with the uniform continuity of $\nabla_{\theta}DF(s,X_s,X(s))$ implies the following estimate for $J^{3,1,1}_\epsilon$: there is a constant $C>0$ independent of $\epsilon$ such that
\begin{align*}
|J^{3,1,1}_\epsilon| \leq C \varpi (\epsilon) \| X_s\|_{L^p} \to 0 \, , \quad \mbox{ as } \epsilon \to 0\, ,
\end{align*}
where $\varpi(\epsilon)$ denotes the modulus of continuity of $\nabla_{\theta} DF(s,\cdot,X(s))$ from Definition \ref{DEF:MC}.
Further, we can formally apply integration by parts to $J^{3,1,2}_\epsilon$ in order to obtain
\begin{align}
J^{3,1,2}_\epsilon =  \frac{1}{\epsilon}DF(s,X_s,X(s))(\beta)\cdot  \int_{\beta-\epsilon}^{\beta}X_{s+\epsilon}(\alpha)d\alpha \Bigg|_{-r}^{0} - \int_{-r}^0 DF(s,X_s,X(s))(\beta)\cdot  \frac{X_{s+\epsilon}(\beta) - X_s(\beta)}{\epsilon} d\beta.
\end{align}
Then it follows that
\begin{align*}
J^{3,1,2}_\epsilon \to  DF(s,X_s,X(s))(-r)\cdot  X(s-r)-DF(s,X_s,X(s))(0)\cdot  X(s) -\Scal{DF(s,X_s,X(s)) \, , \, dX_s},
\end{align*}
as $\epsilon \searrow 0$. Altogether, we have finally shown that
\begin{align*}
\lim_{\epsilon \searrow 0} J^3_\epsilon = \int_0^t \Scal{DF(s,X_s,X(s))\, , \, d X_s}\, .
\end{align*}
This corresponds to \eqref{EQN:ForInt}. Hence, we conclude the proof.
\end{proof}

In Appendix \ref{APP:PC} it is shown, exploiting the results obtained in \cite{f13in}, that the It\^{o} formula \eqref{EQN:Ito} is coherent with the It\^{o} formula for path-dependent processes with jumps proved in \cite{cont2010change}, as well as the results obtained in \cite{f13in}.

Let us consider the \textit{forward integral}
\[
\int_0^t \Scal{DF(s,X_s,X(s))\, , \, d X_s} \, ,
\]
introduced in Theorem \ref{THM:Ito}, we want now to relate the \textit{forward integral} to the operator introduced in \cite{Yan}. In fact, since the right-derivative operator introduced in Definition \ref{DEF:WeRi} is the infinitesimal generator of the left-shift semigroup introduced in \cite{Yan}, it can be shown that the \textit{forward integral} does coincide, under some suitable regularity conditions, with the operator $\mathcal{S}F(s,X_s,X(s))$ introduced in \cite{Yan}.

\begin{proposition}\label{PROP:Conv}
Let $X$ be the solution to equation \eqref{EQN:SFDDE1} and let $F:[0,T] \times L^p \times \R^d \to \R$ be such that $F \in C^{1,1,2}\left ([0,T] \times L^p \times \R^d \right )$ and such that the forward integral defined in Equation \eqref{EQN:ForInt} is well-defined as a limit in probability uniformly on compacts. Additionally, let us assume that $X_s \in W^{1,p}$ for every $s\in [0,T]$. Then
\begin{equation}\label{EQN:WeakG}
\int_0^t \Scal{DF(s,X_s,X(s))\, , \, dX_s} = \int_0^t \Scal{DF(s,X_s,X(s))\, , \, \nabla_\theta^+ X_s} ds  =\int_0^t \mathcal{S}F(s,X_s,X(s)) ds\,
\end{equation}
holds $P$-a.s., where $\mathcal{S}F(s,X_s,X(s))$ is the operator introduced in \cite[Section. 9]{Yan}.
\end{proposition}
\begin{proof}
Let $X_s \in W^{1,p}$. From the fundamental theorem of calculus for absolutely continuous functions we have, $\mathbb{P}$-a.s.,
\begin{align*}
&\lim_{\epsilon \searrow 0}  \Scal{DF(s,X_s,X(s)) \, ,\, \frac{X_{s+\epsilon}-X_s }{\epsilon}- \nabla^+_\theta X_s} \\
&\lim_{\epsilon \searrow 0} \int_{-r}^0DF(s,X_s,X(s))(\beta)\cdot  \left ( \frac{\left (X(s+\epsilon+\beta)-X(s+\beta) \right )}{\epsilon} -\nabla^+_\theta X(s+\beta) \right )d \beta\\
&=\lim_{\epsilon \searrow 0} \int_{-r}^0 DF(s,X_s,X(s))(\beta)\cdot \left (\frac{1}{\epsilon} \int_{s+\beta}^{s+\beta+\epsilon} \nabla_\theta X(r) dr - \nabla^+_\theta X(s+\beta)\right ) d\beta
\end{align*}
Now, by Lebesgue's dominated convergence theorem, which is justified by analogous arguments as for the convergence of \eqref{EQN:J2Etc2}, we finally get
\begin{align*}
\begin{split}
\lim_{\epsilon \searrow 0} \int_{-r}^0 DF(s,X_s,X(s))(\beta) \cdot \left (\frac{1}{\epsilon} \int_{s+\beta}^{s+\beta+\epsilon} \nabla_\theta X(r) dr - \nabla^+_\theta X(s+\beta)\right ) d\beta = 0.
\end{split}
\end{align*}
\end{proof}

Exploiting the standard definition of the Poisson random measure, we can now give another formulation of the It\^{o} formula \eqref{THM:Ito}.

\begin{theorem}[It\^{o}'s formula]\label{THM:Ito2}
Let the hypothesis of Theorem \ref{THM:Ito} hold, then
\begin{equation}\label{EQN:Ito2Jum}
\begin{split}
&F(t,X_t,X(t))=F(0,X_0,X(0)) + \int_0^t \Scal{DF(s,X_s,X(s))\, , \, d X_s} ds +\\
&+ \int_0^t \partial_t F(s,X_s,X(s)) ds + \int_0^t \nabla_x F(s,X_s,X(s))\cdot d X(s) ds+\\
& +\frac{1}{2} \int_0^t \int_0^t Tr \left [  g^*(s,X_s,X(s)) \nabla_x^2 F(s,X_s,X(s))g(s,X_s,X(s))\right ] d s +\\
&+ \sum_{s \leq t} F(s,X_s,X(s))-F(s,X_s,X(s-)) - \Delta X(s) \cdot \nabla_x F(s,X_s,X(s))\, ,
\end{split}
\end{equation}
holds $P$-a.s., where the notation is as in Theorem \ref{THM:Ito} and $\Delta X(s)$ is the jump of the process $X$ at time $s$, namely
\[
\Delta X(s) := X(s) -X(s-)\, .
\]
\end{theorem}
\begin{proof}
It immediately follows from Theorem \ref{THM:Ito} and \cite[Theorem 4.4.10]{MR2512800}.
\end{proof}

Hereinafter, we state a crucial probabilistic property of the solution of SFDDE \eqref{eq:Main1} which is needed for the derivation of Feynman-Kac's formula also stated below. As it is perceptible, the finite-dimensional process ${}^{(\eta,x)} X(t)$, $t\in [0,T]$, $(\eta,x) \in M^p$ is not Markov, since the value of ${}^{(\eta,x)} X(t)$ depends on the \emph{near} past. Nevertheless, if we enlarge the state space and regard the process $X$ as a process of the segment, i.e. with infinite-dimensional state space, in the present case $M^p$, then such process is indeed Markovian.

The proof follows almost immediately given the fact that the Markov property of the solution is fully known for the case without jumps, i.e. $h=0$, see \cite[Theorem (III. 1.1)]{SMohammed2}, which actually follows from measure theoretical properties of the driving noise and not path or distributional properties of it. More concretely, one would expect the Markov property of the solution to hold if, for instance, the driving noises have independent increments which is the case in our setting.

In order to state the Markov property one is compelled to look at solutions starting at time $t_1\geq 0$, that is we hereby consider ${}^{(t_1,\eta,x)} X_t$, $t\geq t_1$, $(\eta,x) \in M^p$, the segment of the solution starting in $(\eta,x)$ at times $t_1\geq 0$, i.e.
\[
\begin{split}
{}^{(t_1,\eta,x)} X (t) &= \eta ( 0) + \int_{t_1}^t f(s,{}^{(t_1,\eta,x)} X_s,{}^{(t_1,\eta,x)} X (s))ds + \int_{t_1}^t g(s,{}^{(t_1,\eta,x)} X_s, {}^{(t_1,\eta,x)} X (s)) dW(s) \\
&+ \int_{t_1}^t \int_{\R_0} h(s,^{(t_1,\eta,x)} X_s,{}^{(t_1,\eta,x)} X (s),z)\widetilde{N}(ds,dz)\, ,\\
\end{split}
\]
for every $t\in [t_1,T]$ and ${}^{(t_1,\eta,x)} X (t) = \eta (t-t_1)$ for every $t\in [t_1-t,t_1)$. Define further the family of operators
\begin{align*}
T_{t}^{t_1}: L^2(\Omega,\mathcal{F}_{t_1}; M^p) &\longrightarrow L^2(\Omega, \mathcal{F}_{t}; M^p)\\
(\eta,x) &\longmapsto \left ( {}^{(t_1,\eta,x)} X_t,{}^{(t_1,\eta,x)} X (t) \right ) \, .
\end{align*}
We denote $T_t = T_t^0$. It turns out that, under hypotheses $(\mathbf{L_1})$ and $(\mathbf{L_2})$, $T_t^{t_1}$ is Lipschitz continuous and the family of operators $\{T_t^{t_1}\}_{0\leq t_1\leq t \leq T}$ defines a semigroup on $L^2(\Omega, M^p)$, i.e.
$$T_{t_2} (\eta,x) = T_{t_2}^{t_1} \circ T_{t_1} (\eta,x),$$
for every $0\leq t_1\leq t_2\leq T$ and $(\eta,x) \in L^2(\Omega,\mathcal{F}_{0}; M^p)$. The latter property can easily be obtained by showing that both sides of the identity solve the same SFDDE and the fact that solutions are unique, see \cite[Theorem (II. 2.2)]{SMohammed2} for the case $h=0$.

\begin{theorem}[The Markov property]\label{THM:MP}
Assume hypotheses $(\mathbf{L_1})$ and $(\mathbf{L_2})$ hold and ${}^{(\eta,x)}X(t)$, $t\in [0,T]$, $(\eta,x) \in M^p$ denotes the unique strong solution of the SFDDE \eqref{eq:Main1}. Then the random field
$$\left \{\left ( {}^{(\eta,x)} X_t,{}^{(\eta,x)} X (t) \right ) : \, t\in [0,T], (\eta,x) \in M^p \right \}$$
describes a Markov process on $M^p$ with transition probabilities given by
$$p(t_1,(\eta,x),t_2, B) = \mathbb{P}(\omega \in \Omega, \, T_{t_2}^{t_1}(\eta,x)(\omega) \in B),$$
for $0\leq t_1\leq t_2\leq T$, $(\eta,x) \in M^p$ and Borel set $B\in \mathcal{B}(M^p)$. In other words, for any $(\eta,x) \in L^2(\Omega, \mathcal{F}_0; M^p)$ and Borel set $B\in \mathcal{B}(M^p)$, the Markov property
$$\mathbb{P}(T_{t_2}(\eta,x)\in B |\mathcal{F}_{t_1}) = p(t_1, T_{t_1}(\eta,x),t_2,B) = \mathbb{P}(T_{t_2}(\eta,x) \in B|T_{t_1}(\eta,x))$$
holds a.s. on $\Omega$.
\end{theorem}
\begin{proof}
we can see that for every $0\leq t_1\leq t_2\leq T$ and every $(\eta,x) \in M^p$, the mapping $B\mapsto p(t_1,(\eta,x),t_2,B)= \mathbb{P} \circ (T_{t_2}^{t_1}(\eta,x))^{-1}(B)$, $B\in \mathcal{B}(M^p)$ defines a probability measure on $M^p$, since the random variable $T_{t_2}^{t_1}(\eta,x):\Omega \rightarrow M^p$ is $(\mathcal{F},\mathcal{B}(M^p))$-measurable. We would then like to show that, if $0\leq t_1\leq t_2\leq T$, then
\begin{align}\label{Markov1}
\mathbb{P}\Big(\omega\in \Omega: \, T_{t_2}(\eta,x)(\omega)\in B|\mathcal{F}_{t_1}\Big)(\omega') = p(t_1,T_{t_1}(\eta,x)(\omega'),t_2,B)
\end{align}
for almost all $\omega'\in \Omega$, every Borel set $B\in \mathcal{B}(M^p)$ and any $(\eta,x) \in L^2(\Omega,\mathcal{F}_0;M^p)$. The right-hand side of \eqref{Markov1} equals
$$\int_{\Omega} 1_{B}\bigg( \left( T_{t_2}^{t_1}(T_{t_1}(\eta)(\omega')) \right)(\omega) \bigg) \mathbb{P}(d\omega)$$
for almost all $\omega'\in \Omega$. Hence, by the definition of conditional expectation, identity \eqref{Markov1} is synonymous with
\begin{align}\label{Markov2}
\int_A 1_{B} \left( T_{t_2}(\eta,x)(\omega)\right) \mathbb{P}(d\omega) = \int_A \int_{\Omega} 1_{B}\left( \left( T_{t_2}^{t_1}(T_{t_1}(\eta,x)(\omega')) \right)(\omega) \right) \mathbb{P}(d\omega) \mathbb{P}(d\omega')
\end{align}
for all $A\in \mathcal{F}_{t_1}$ and all $B\in \mathcal{B}(M^p)$. In a summary, the main challenge is to verify relation \eqref{Markov2} which is stated in quite general terms.

Let $\mathcal{G}_t$, $t\in [0,T]$ be the $\sigma$-algebra generated by $\{N((s,u],B), t< s\leq u\leq T, B\in \mathcal{B}(\R_0)\}$. The key steps in proving \eqref{Markov2} according to \cite[Theorem (III. 1.1)]{SMohammed2} are the following. First, the family of operators $\{T_t\}_{t\in [0,T]}$ defines a semigroup on $L^2(\Omega, M^p)$. Secondly, the $\sigma$-algebras $\mathcal{F}_{t}$ and $\mathcal{G}_{t}$ are independent for every $t\in[0,T]$, and $\{T_t^{t_1}(\eta,x)\}_{t\geq t_1}$ is adapted to $\{\mathcal{F}_t\cap \mathcal{G}_{t_1}\}_{t\geq t_1}$, being each $\mathcal{F}_t\cap \mathcal{G}_{t_1}$ independent of $\mathcal{F}_{t_1}$. Finally, a key point to prove \eqref{Markov2} is that one can first prove
\begin{align}\label{Markov3}
\int_A f \left( T_{t_2}(\eta,x)(\omega)\right) \mathbb{P}(d\omega) = \int_A \int_{\Omega} f\left( \left( T_{t_2}^{t_1}(T_{t_1}(\eta,x)(\omega')) \right)(\omega) \right)  \mathbb{P}(d\omega)  \mathbb{P}(d\omega')
\end{align}
for any bounded continuous function $f:M^p\rightarrow \R$. Then one can use a limit argument to show the relation \eqref{Markov2} for the case $f=1_{B}$ being $B$ just an open set of $M^p$ and eventually for any general indicator function on Borel sets. The argument which transfers \eqref{Markov3} into the case $f=1_B$ for any open set $B$ in $M^p$ requires that the state space in consideration is separable so that $1_{B}$, being $B$ an open set of $M^p$, can be approximated by uniformly continuous partitions of unity $\{f_m\}_{m\in \mathbb{N}}$, $f_m:M^p\rightarrow \R$ such that $\displaystyle\lim_{m\to \infty} f_m = 1_B$. All these properties above mentioned are indeed satisfied in our framework.
\end{proof}

Exploiting It\^{o}'s formula from Theorem \ref{THM:Ito} together with the Markov property from Theorem \ref{THM:MP}, we can now prove a Feynman-Kac theorem for $M^p-$valued SFDDE's with jumps.

\begin{theorem}[Feynman-Kac theorem]\label{THM:FK}
Let the hypothesis of Theorem \ref{THM:Ito} hold, then the following \textit{path-dependent partial integro-differential equation} (PPIDE) holds
\begin{equation}\label{EQN:FK}
\begin{cases}
\partial_t F(t,\eta,x) &= \Scal{DF(s,\eta,x)\, , \, d \eta} + \nabla_x F(t,\eta,x)  \cdot f(t,\eta) \\
&+ \frac{1}{2} Tr \left [g(t,\eta,x) g^*(t,\eta,x) \nabla^2_x F(t,\eta,x)\right ]\\
& + \int_{\R_0}\left ( F(t,\eta,x +h(t,\eta,x)(z)) - F(t,\eta,x) -\nabla_x F(t,\eta,x) h(t,\eta,x)(z)\right ) \nu(dz)\, ,\\
F(T,\eta,x) &= \Phi(\eta,x)\, ,
\end{cases}
\end{equation}
with $\Phi \in C^{1,2 }( L^p \times \R^d)$, then we have
\begin{equation}\label{EQN:Fe}
F(t,\eta,x) := \mathbb{E}\left [\left . \Phi(X_T,X(T)) \right | X_t = \eta, \, X(t) = x\right ]\, ,\quad t \in [0,T]\, ,
\end{equation}
where $(X_t,X(t))$ solves the SFDDE \eqref{EQN:SFDDE1}. If further $\Phi\in C^{1,2} (L^p \times \R^d)$, then the converse holds true, as well.
\end{theorem}
\begin{proof}
We have to show that if a function $F:[0,T] \times L^p \times \R^d \to \R$ satisfies the PIDE \eqref{EQN:FK}, then we have that equation \eqref{EQN:Fe} holds. Let us assume $X$ is the unique solution to equation \eqref{EQN:SFDDE1}, as in Section \ref{TheEquationsExistenceUniquenessMomentestimates} we will use the notation $^{(\tau,\eta,x)} X$ to denote the process with initial time $\tau \in [0,T]$ and initial value $(\eta,x) \in M^p$. If $F$ satisfies equation \eqref{EQN:FK} by It\^{o}'s formula \eqref{EQN:Ito} we have
\begin{equation}\label{EQN:Rhs}
\begin{split}
&F(T,^{(\tau,\eta,x)}X_T,^{(\tau,\eta,x)}X(T)) -F(\tau,\eta,x) \\
&=\int_\tau^T \nabla_x F(s,^{(\tau,\eta,x)}X_s,^{(\tau,\eta,x)}X(s)) \cdot g(s,^{(\tau,\eta,x)}X_s,^{(\tau,\eta,x)}X(s)) dW(s) \\
&+\int_\tau^T \int_{\R_0} \left (F(s,^{(\tau,\eta,x)}X_s,^{(\tau,\eta,x)}X(s) + h(s,^{(\tau,\eta,x)}X_s,^{(\tau,\eta,x)}X(s))(z)\tilde{N}(ds,dz)\right .\\
 &-\int_\tau^T \int_{\R_0} \left . F(s,^{(\tau,\eta,x)}X_s,^{(\tau,\eta,x)}X(s))\right ) \tilde{N}(ds,dz)\, .
\end{split}
\end{equation}
Taking now the expectation, exploiting the fact that the right-hand side of equation \eqref{EQN:Rhs} has null conditional expectation and using the terminal condition we have for any $0 \leq \tau < t \leq T$,
\[
F(t,\eta,x) = \mathbb{E}\left [\left . \Phi(X_T,X(T)) \right | X_t = \eta, \, X(t) = x\right ]\, ,
\]
and the first implication is proved.

Conversely, let us now suppose equation \eqref{EQN:Fe} holds true, then from the Markov property from Theorem \ref{THM:MP} of the $M^p$-valued process and the tower rule for the conditional expectation, we have that for $0 \leq \tau < t \leq T$,
\[
\begin{split}
&\mathbb{E} \left [\left .F(t,X_t,X(t))-F(\tau,X_\tau,X(\tau))\right | X_\tau = \eta ,X(\tau) = x\right ]= \\
&=\mathbb{E} \left [\left .\mathbb{E}\left [ \left .\Phi(X_T,X(T))\right |X_t,X(t)\right ]-\mathbb{E}\left [ \left .\Phi(X_T,X(T))\right |X_\tau,X(\tau)\right ]\right | X_\tau,X(\tau)\right ] =\\
&=\mathbb{E} \left [\left .\Phi(X_T,X(T))\right | X_\tau,X(\tau)\right ]-\mathbb{E} \left [\left .\Phi(X_T,X(T))\right | X_\tau,X(\tau)\right ]=0\, .
\end{split}
\]
On the other side, we can apply It\^{o}'s formula (Theorem \ref{THM:Ito}) to the function $F$, then we have that for $0 \leq \tau < t \leq T$,
\[
\begin{split}
&F(t,X_t,X(t))=F(\tau,X_\tau,X(\tau)) + \int_\tau^t \Scal{DF(s,X_s,X(s))\, , \, d X_s}  + \int_\tau^t \partial_t F(s,X_s,X(s)) ds \\
&+ \int_\tau^t \nabla_x F(s,X_s,X(s))\cdot  d X(s) +\frac{1}{2} \int_\tau^t Tr \left [g^*(s,X_s,X(s)) \nabla_x^2 F(s,X_s,X(s))g(s,X_s,X(s))\right ] d s \\
&+ \int_\tau^t \int_{\R_0} \left (F(s,X_s,X(s)+h(s,X_s,X(s))(z))-F(s,X_s,X(s)) \right ) \nu(dz)\,ds\\
&- \int_\tau^t \int_{\R_0}   \nabla_x F(s,X_s,X(s))h(s,X_s,X(s))(z) \nu(dz)\,ds \\
 &+ \int_\tau^t \int_{\R_0} \left (F(s,X_s,X(s)+h(s,X_s,X(s))(z)) - F(s,X_s,X(s))\right ) \tilde{N}(ds,\,dz)\, .
\end{split}
\]
Then taking conditional expectation the PPIDE \eqref{EQN:FK} holds true.
\end{proof}

\appendix

\section{Appendix}
\subsection{Kunita's inequality}\label{sec:Kunita}
In Section \ref{sec:functionalsOnS},
 we introduced a general version of \emph{Kunita's inequality}, (Corollary 2.12 in \cite{MR2090755}). For $n=1$, this is a rewritten version of Corrolary 2.12 in \cite{MR2090755}). Below, we explain how to extend the result to general $n$.
\begin{proof}[Proof of Lemma \ref{lem:Kunita}]
 Notice that since norms on $\mathbb R^n$ are equivalent, it holds that
\begin{align*}\sum_{j=1}^n &|a_j|^{q}\leq C_0\Big(\sum_{j=1}^n|a_j|\Big)^q,\textnormal{ and}\\
\big(\sum_{j=1}^n&|a_j|\big)^{q}\leq C_1\sum_{j=1}^n |a_j|^q,
\end{align*}
for some constants $C_0, C_1$ depending only on $n$ and $q$. We may assume that $C_0>1$
\begin{align*}
\sum_{j=1}^n &\| H^{,j}(s)\|^q_{L^2(\nu_j, \mathbb R^d)}=\sum_{j=1}^n \big(\int_{\mathbb R_0}|H^{,j}(s,z)|^2\nu_j(dz)\Big)^{\frac{q}{2}}\\
& \leq C_0\big(\sum_{j=1}^n \int_{\mathbb R_0}|H^{,j}(s,z)|^2\nu_j(dz)^{\frac{1}{2}}\Big)^q=C_0 \| H(s)\|^q_{L^2(\nu, \mathbb R^d)}.
\end{align*}
Then, if we write out "the columns wise" form of the $\tilde N$-integral, we obtain 
\begin{align*}
\sup_{0\leq u\leq t}&\Big|\sum_{j=1}^n \int_0^u \int_{\mathbb R_0}H^{,j}(s,z)\tilde N(ds,dz)\Big|^q \leq n^{q-1}\sup_{0\leq u\leq t}   \sum_{j=1}^n\Big| \int_0^u \int_{\mathbb R_0}H^{,j}(s,z)\tilde N(ds,dz)\Big|^q       \\
&\leq n^{q-1}\sum_{j=1}^n \int_0^t \| H^{,j}(s)\|_{L^q(\Omega, L^q(\nu_j,\mathbb R^{d}))} +  \| H^{,j}(s)\|_{L^q(\Omega, L^2(\nu_j,\mathbb R^{d}))} \,ds\\
&=n^{q-1}\int_0^t E\Big[\sum_{j=1}^n \| H^{,j}(s)\|^q_{L^q(\nu_j, \mathbb R^d)}+\sum_{j=1}^n \| H^{,j}(s)\|^q_{L^2(\nu_j, \mathbb R^d)}\Big]ds\\
&\leq n^{q-1}\int_0^t E\Big[\| H(s)\|^q_{ L^q(\nu, \mathbb R^{d\times n})}+ C_0 \| H(s)\|^q_{L^2(\nu, \mathbb R^{d\times n})}\Big]ds\\
&\leq n^{q-1}C_0 \int_0^t  \| H(s)\|^q_{L^q(\Omega, L^q(\nu, \mathbb R^{d\times n}) )}+ \| H(s)\|^q_{L^q(\Omega, L^2(\nu, \mathbb R^{d\times n}) )} ds.
\end{align*}
\end{proof}

\subsection{Connection with path-dependent calculus}\label{APP:PC}

In what follows we provide a connection between It\^{o}'s formula \eqref{THM:Ito} and the path-dependent It\^{o}'s formula given in \cite{cont2010change,cont2013functional} which relies on the concepts of vertical and horizontal derivative, there introduced. Let us first set the notation we  use in the current section.

Let $(\Omega, \mathcal{F},P)$ be the probability space with $\Omega=\mathcal{D}([0,T],\R^d)$ endowed with the $P$-augmented (right-continuous) filtration $\{\mathcal{F}_t\}_{t\in [0,T]}$ generated by the canonical process $Y:[0,T]\times\Omega \rightarrow \R^d$, $Y(t,\omega) = \omega(t)$ and here $\mathcal{F}:=\mathcal{F}_T$. In this setting we define, for every $\omega\in \Omega$ and $t\in [0,T]$, $\omega_t :=\{\omega(s), \, 0\leq s\leq t\}\in \mathcal{D}([0,t])$, the trajectory up to time $t$. A stochastic process is a function $\varphi:[0,T]\times \Omega\rightarrow \R^d$, $(t,\omega)\mapsto \varphi(t,\omega)$. In addition, we say $\varphi$ is non-anticipative if it is defined on $\mathcal{D}([0,t];\R^d)$, i.e. $\varphi(t,\omega)=\varphi(t,\omega_t):=\varphi_t(\omega_t)$.

Let $\varphi=\{\varphi_t, t\in [0,T]\}$ be a non-anticipative stochastic process and $\{e_i\}_{i=1}^d\subset \R^d$ the canonical basis, we define the so-called \emph{vertical derivative} as the following (path-wise) limit

$$\mathcal{D}^{V,i} \varphi_t (\omega_t) = \lim_{h\to 0}\frac{\varphi_t(\omega_t^{h e_i})-\varphi_t(\omega_t)}{h},$$
where $\omega_t^{he_i}(s):= \omega_t(s)+he_i 1_{\{t\}}(s)$, for every $s\in [0,t]$. Here, $\omega_t^{he_i}$ means adding a jump of size $h$ at time $t$ on the direction of $e_i$ and hence the name. We then define the \emph{vertical gradient} of $\varphi_t$ as
$$\mathcal{D}^V \varphi_t = \left( \mathcal{D}^{V,1}\varphi_t, \dots, \mathcal{D}^{V,d}\varphi_t \right).$$
Furthermore, we define the \emph{horizontal derivative} as the following (path-wise) limit
$$\mathcal{D}^{H} \varphi_t (\omega_t) = \lim_{h\searrow 0} \frac{\varphi_{t+h}(\omega_{t,h})-\varphi_t(\omega_t)}{h},$$
where $\omega_{t,h}(s):= \omega_t(s)1_{[0,t]}(s)+\omega_t(t)1_{(t,t+h]}(s)$, for every $s\in [0,t+h]$. Here, $\omega_{t,h}$ is the extension of the trajectory $\omega_t$ on $[0,t]$ to $[0,t+h]$ by an horizontal line of length $h$ at $\omega_t(t)$ and hence the name.

We consider a functional $F:[0,T]\times \mathcal{D}([0,T];\R^d)\rightarrow \R$ which will act on processes $\varphi_t$. We say $F$ is non-anticipative if
$$F(t,\psi)=F(t,\psi_t)=: F_t(\psi_t),$$
for every non-anticipative stochastic process $\psi_t$. Next, we state an It\^{o} formula for $F_t(\psi_t)$ where $F_t$ is a non-anticipative functional which is once horizontally and twice vertically differentiable. This result is taken from \cite[Proposition 6]{cont2010change}.

\begin{theorem}[Functional It\^{o}'s formula]\label{functional Ito formula}
Consider an $\R^d$-valued non-anticipative stochastic process $\varphi_t$ which admits the following c\`{a}dl\`{a}g semimartingale representation
$$\varphi_t = \varphi_0 + \int_0^t \mu(s) ds + \int_0^t \sigma(s)dW(s)+\int_0^t \int_{\R_0} \gamma(s-,z)\widetilde{N}(ds,dz)$$
for processes $\mu:[0,T]\rightarrow \R^d$, $\sigma:[0,T]\rightarrow \R^{d\times m}$ and $\gamma:[0,T]\times \R_0 \rightarrow \R^{d\times n}$ such that $\int_0^T E\left[|\mu(s)|+\|\sigma(s)\|^2+\int_{\R_0}\|\gamma(s,z)\|^2 \nu(dz)\right] ds <\infty$ being $\|\cdot\|$ a matrix norm. 

Let $F$ be a once horizontally and twice vertically differentiable non-anticipative functional satisfying some technical continuity conditions on $F$ (see \cite[Proposition 6]{cont2010change}), $\mathcal{D}^V F_t$, $\mathcal{D}^V \mathcal{D}^V F_t$ and $\mathcal{D}^H F_t$. Then for any $t$ the following functional It\^{o} formula holds $P$-a.s.
\begin{align}\label{EQN:ItoLevy}
\begin{split}
\displaystyle F_t\left(\varphi_t\right) &=F_0\left( \varphi_0\right) +\int_{(0,t]}\DDD^H F_s(\varphi_{s_-})ds+ \int_{(0,t]} \DDD^V F_s(\varphi_{s_-}) \, dX(s) \\
&+ \int_{(0,t]} \frac{1}{2} Tr\,\left[\sigma^*(s) \DDD^{V}\DDD^{V} F_s(\varphi_{s_-})\sigma(s)\right] ds \\
&+ \int_{(0,t]} \int_{\R_0} \DDD^V F_s(\varphi_{s_-}) \left (F_s(\varphi_{s_-} + \gamma(s-,z)1_{\{s\}})-F_s(\varphi_{s_-}) -\gamma(s-,z) \right )N(ds,dz) \, .
\end{split}
\end{align}

\end{theorem}

To be able to show that the path-dependent It\^{o}'s formula \eqref{EQN:ItoLevy} and the It\^{o}'s formula from Theorem \ref{THM:Ito} do coincide we need first to connect the two settings. Such a link can be established following \cite{f13in}, where the following operators are considered:
\begin{itemize}
 \item
the \textit{restriction} operator, for every $t\in [0,r]$
\[
\begin{split}
&M_t : \mathcal{D}([-r,0],\R^d)\to \mathcal{D}([0,t],\R^d) \, ,\\
&M_t (f)(s) = f(s-t)\, ,\quad s \in [0,t)\, ,
\end{split}
\]
\item
the \textit{backward extension} operator, for every $t\in (0,r)$
\[
\begin{split}
&L_t : \mathcal{D}([0,t],\R^d) \to  \mathcal{D}([-r,0],\R^d) \, ,\\
&L_t (f)(s) = f(0)\IndN{[-r,-t)}(s) + f(t+s) \IndN{[-t,0)}(s)\, ,\quad s \in [-r,0)\, ,
\end{split}
\]
\end{itemize}

Let us consider a non-anticipative functional $b:[0,T]\times \mathcal{D}([0,T];\R^d)\rightarrow \R$, $b(t,\psi)=b(t,\psi_t)=:b_t(\psi_t)$ for any non-anticipative stochastic process $\psi$, then one can define a different functional $\widehat{b}$ on $[0,T] \times \mathcal{D}([-r,0];\R^d)\times \R^d$ as 
\begin{align*}
\widehat{b}(t,X_t,X(t)) := b_t\left (\tilde{M}_t X_t\right )\, ,\quad (X_t,X(t)) \in \mathcal{D}([-r,0];\R^d)\times \R^d\, ,
\end{align*}
with 
\[
\tilde{M}_t X_t (s) :=
\begin{cases}
M_t (X_t)(s) & \mbox{ if } \, s \in [0,t)\,\\
X_t(s) & \mbox{ if } \, s=t \,\\
\end{cases}
\, .
\]
The converse holds true as well, in fact let us consider a given functional $\widehat{b}$ on $[0,T] \times \mathcal{D}([-r,0];\R^d)$, then we can obtain a corresponding functional $ b_t$ on $\mathcal{D}([0,t];\R^d)$ as
\begin{equation}\label{EQN:MExt}
b_t(\varphi_t) := \widehat{b}(t,L_t \varphi_t,\varphi_t(t))\, ,\quad  \varphi_t \in \mathcal{D}([0,t];\R^d)\, ,
\end{equation}
see \cite{f13in} for details.

We can now show how the vertical and horizontal derivatives can be written in terms of the Fr\'{e}chet derivative $D$ and the derivative with respect to the present state. Part of the next theorem was already established in \cite[Theorem 6.1]{f13in}.

\begin{proposition}\label{THM:RepPathNoP}
Consider a function $F :[0,T] \times \mathcal{D}([-r,0];\R^d) \to \R $ and let us define $u_t :\mathcal{D}([0,t];\R^d) \to \R $ as above in \eqref{EQN:MExt} $ u_t (X_t) := F(t,L_t X_t,X(t))$. Then the $i$-th vertical derivative $ \DDD^{V,i} $ of $u_t$ coincides with the derivative with respect to the present state $X^i(t) $ of $F$, namely 
\begin{equation}\label{EQN:VertDer}
\DDD^{V,i} u_t(X_t)= \partial_{x_i} F(t,L_t X_t,X(t))\, .
\end{equation}
Furthermore, we have
\begin{equation}\label{EQN:JumpPS}
u_t(X_t^{h^i})-u_t(X_t) = F(t,L_t X_t,X(t)+h^i(t,L_t X_t,X(t))-F(t,L_t X_t,X(t))\, .
\end{equation}
If we assume that $X_t \in W^{1,p}$, then
\begin{equation}\label{EQN:Levy}
\DDD^H u_t (X_t) = \partial_t  F(t,L_t X_t,X(t)) + \ScalD{D F(t,L_t X_t,X(t)),\nabla^+_\theta L_tX_t}\, ,
\end{equation}
holds, where the notation is given in Section \ref{SEC:Ito}. 
\end{proposition}
\begin{proof}
Concerning \eqref{EQN:VertDer} we have 
\begin{equation}\label{EQN:DimHo}
\begin{split}
\DDD^{V,i} u_t(X_t) &= \lim_{h \to 0} \frac{1}{h} \left( u_t (X_t^h) -u_t(X_t) \right ) = \lim_{h \to 0} \frac{1}{h} \left( F(t,L_t X_t^h,X^h(t)) -F(t,L_t X_t,X(t))\right)  \\
&=  \lim_{h \to 0} \frac{1}{h} \left( F(t,X(t)+h,L_t X_t^h) -F(t,X(t),L_t X_t) \right ) =  \partial_i F(t,X(t),L_t X_t)\, .\\
\end{split}
\end{equation}
For what concerns \eqref{EQN:JumpPS}, proceeding as in \eqref{EQN:DimHo}, we immediately have
\[
\begin{split}
u_t(X_t^{h^i})-u_t(X_t)  &= F(t,L_t X_t^{h^i},X^{h^i}(t)) -F(t,L_t X_t,X(t)) = \\
&=F(t,X(t)+h^i,L_t X_t^{h^i}) -F(t,X(t),L_t X_t)\, .
\end{split}
\]
We refer to \cite[Theorem 6.1]{f13in} for a proof of equation \eqref{EQN:Levy}
\end{proof}

In the framework of this section, exploiting the previous proposition we have that, for suitable regular coefficients, It\^{o}'s formula from Theorem \ref{THM:Ito} and the path-dependent It\^{o}'s formula in Theorem \ref{functional Ito formula} coincide. In particular let us consider a process $X$ evolving according to
\begin{equation}\label{EQN:L2SM}
\begin{cases}
dX_t = f(t,X_t) dt + g(t,X_t) dW(t) + \int_{\R_0} h(t,X_t,z) \tilde{N}(dt,dz)\, ,\\
X_0 = \eta\, ,
\end{cases}
\end{equation}
for some suitably regular enough coefficients $f$, $g$ and $h$. Then proceeding as above we have that Equation \eqref{EQN:L2SM} can be written as a path dependent process
\begin{align*}
\begin{cases}
dX_t = \hat{f}_t(X_t) dt + \hat{g}_t(X_t) dW(t) + \int_{\R_0} \hat{h}_t(X_t) \tilde{N}(dt,dz)\, ,\\
X_0 = \eta\, ,
\end{cases}
\end{align*}
with $\hat{f}_t$, $\hat{g}_t$ and $\hat{h}_t$ defined as in \eqref{EQN:MExt}. Then we have the following result.

\begin{theorem}
Let $F:[0,T] \times M^p \to \R$, $F \in C^{1,1,2}([0,T] \times \mathcal{D}\times \R^d)$ and let us define $u_t :\mathcal{D}([0,t];\R^d) \to \R $ as in \eqref{EQN:MExt} $ u_t (X_t) := F(t,L_t X_t,X(t))$. Then It\^{o}'s formula from Theorem \ref{THM:Ito} and the path dependent It\^{o}'s formula from Theorem \ref{functional Ito formula} coincide.
\end{theorem}
\begin{proof}
It is straightforward from Proposition \ref{THM:RepPathNoP} exploiting the backward extension operator $L_t$ and eventually using It\^{o}'s formula from Theorem \ref{THM:Ito} and the \textit{path-dependent} It\^{o} formula from Theorem \ref{functional Ito formula}.
\end{proof}

\renewcommand{\abstractname}{Acknowledgements}
\begin{abstract}
At its early stage, this research has benefit of the sponsorship of the program Stochastics in Environmental Finance and Economics (SEFE) hosted at and funded by the Centre of Advanced Studies (CAS) of the Norwegian Academy of Science and Letters in the year 2014/15. CAS is thanked for its generous support and nice working environment. This research was completed within the NFR project FINEWSTOCH which is gratefully acknowledged.
\end{abstract}

\nocite{*}
\nocite{SMohammed2}
\nocite{MR0226684}
\bibliographystyle{abbrv}

\bibliography{jumpDelayBibliography}

\end{document}